\documentclass[reqno,11pt]{amsart}
\pdfoutput=1

\usepackage[utf8]{inputenc}
\usepackage[italian,english]{babel}
\usepackage{amsfonts}
\usepackage{amssymb}
\usepackage{tikz}
\usepackage{svg}
\usepackage[right=3cm,bottom=2.5cm,footskip=30pt]{geometry}
\usepackage[autostyle]{csquotes}
\usepackage{color}
\usepackage{mathrsfs}
\usepackage{mathtools,leftidx}
\usepackage{enumitem}
\usepackage{mathrsfs}
\usepackage{accents}
\setlist[enumerate,2]{label={(\theenumi.\theenumii)},ref={(\theenumi.\theenumii)}}

\usepackage{subcaption}

\usetikzlibrary{svg.path,arrows}
\tikzset{>=latex'}
\definecolor{ColOrange}{HTML}{E27D60}
\definecolor{ColBlue}{HTML}{85CDCA}
\definecolor{ColYellow}{HTML}{E8A87C}
\definecolor{ColPink}{HTML}{C38D9D}
\definecolor{ColGreen}{HTML}{40B3A2}

\usepackage[toc]{appendix}
\usepackage{etoolbox}
\cslet{blx@noerroretextools}\empty
\usepackage[maxbibnames=9,giveninits=true,style=ext-alphabetic,backend=biber]{biblatex}

\makeatletter
 \def\author@andify{
 \nxandlist {\unskip{} $\cdot$ \penalty-2}
 {\unskip {} $\cdot$ \penalty-2}
 {\unskip {} $\cdot$ \penalty-2}}
\makeatother

\newcommand{\orcid}[1]{\unskip {} \raisebox{-.3ex}{\href{https://orcid.org/#1}{\includegraphics{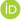}}}}

\DeclareFieldFormat{titlecase:title}{\MakeSentenceCase*{#1}}
\renewbibmacro*{url+urldate}{%
\usebibmacro{url}\ifentrytype{online}
{\addspace\usebibmacro{urldate}}{}}
\usepackage[colorlinks, citecolor=citegreen, linkcolor=refred,urlcolor=blue, unicode,psdextra,hypertexnames=false]{hyperref}
\usepackage{cleveref}
\crefname{enumi}{}{}
\crefname{enumii}{}{}
\expandafter\def\csname ver@etex.sty\endcsname{3000/12/31}

\usepackage{accents}
\usepackage{autonum}
\usepackage{gasymbols}

\definecolor{citegreen}{rgb}{0,0.3,0}
\definecolor{refred}{rgb}{0.5,0,0}

\def\mathring#1{\accentset{\circ}{#1}}
\usepackage[equation=section]{theorems}

\let\oldemail\email
\let\email\relax
\def\email#1{\oldemail{\href{mailto:#1}{\textcolor{black}{#1}}}}

\title[Isoperimetric sets and mass]{Isoperimetric sets in nonnegative scalar curvature and their role through various concepts of mass}

\author[L.~Benatti]{Luca Benatti\orcid{0000-0002-4685-7443}}
\address{L.~Benatti, Universit\`a degli Studi di Pisa,
Largo Bruno Pontecorvo 5, 56127 Pisa, Italy}
\email{luca.benatti@dm.unipi.it}

\author[M.~Fogagnolo]{Mattia Fogagnolo\orcid{0000-0002-5933-1344}}
\address{M.~Fogagnolo, Universit\`a di Padova, via Trieste 63, 35121 Padova (PD), Italy}
\email{mattia.fogagnolo@unipd.it}

\renewcommand{\ncapa}{\mathfrak{c}}
\newcommand{\eum}{\delta}
\newcommand{\ma}{\mathfrak{m}}
\newcommand{\tp}[1][p]{{\text{\tiny$(#1)$}}}

\begin{document}
 
\begin{abstract}
We review some recent results about the relations among isoperimetric sets, Penrose inequalities and related concepts in the analysis of $3$-manifolds of nonnegative scalar curvature. We also show that if the isoperimetric sets of big volume have connected boundaries, the equivalence among suitable notions of mass hold.  
\end{abstract}
\maketitle


\medskip

\noindent \underline{\smash{Keywords}}: isoperimetric sets, Penrose inequality, positive mass theorem, isoperimetric mass, isocapacitary mass, inverse mean curvature flow, nonnegative scalar curvature.
\section{Introduction}
{We will discuss topics gravitating around isoperimetric properties of Riemannian manifolds of dimension $3$ with nonnegative scalar curvature that are asymptotically flat in some suitable sense, usually endowed with a closed, minimal and outermost boundary. With the latter adjective, we indicate that no other closed, minimal surface exists enclosing $\partial M$. We will occasionally refer to boundaries with these properties with the word \emph{horizon}, or \emph{horizon boundary}.

One of the classical results in this class of manifolds is the Riemannian Penrose inequality.  Leaving all the discussion and the details to the main body of the work, we just point out that Penrose inequalities read as a bound from above of the area of the minimal boundary of $(M, g)$ in terms of suitable global geometric invariants, which are interpreted as physical ``global" \emph{masses}. In \cite{benatti_isoperimetricriemannianpenroseinequality_2022}, together with Mazzieri, we showed an isoperimetric version of the Riemannian Penrose inequality holding in a very large class of manifolds. We will review such  result focusing on its relation with the isoperimetric sets and on the techniques that led to the proof of their existence for any volume in this context \cite{carlotto_effectiveversionspositivemass_2016}. We will also deal with the equivalence of various, apparently very different, notions of mass. Particular attention will be put in being as sharp as possible in the decay requirements for the asymptotically flat condition on the manifolds considered.}

\smallskip

In \cref{sec:inverse}, we review the main properties of the Inverse Mean Curvature Flow we are going to employ, mainly obtained by Huisken-Ilmanen \cite{huisken_inversemeancurvatureflow_2001}. In \cref{sec:isoperimetry}, we review and detail the beautiful proof of the existence of isoperimetric sets of any volume in $3$-manifolds with nonnegative scalar curvature, obtained by Carlotto-Chodosch-Eichmair \cite{carlotto_effectiveversionspositivemass_2016}, after the fundamental insight of Shi \cite{shi_isoperimetricinequalityasymptoticallyflat_2016}. In doing so, we sensibly weaken the decay assumptions on the metric to $\CS^0$-asymptotic flatness. In \cref{sec:penrose}, after having discussed Huisken's notion of isoperimetric mass \cite{huisken_isoperimetricconceptmassgeneral_2006}, we review the proof of the related isoperimetric Penrose inequality, obtained in collaboration with Mazzieri \cite{benatti_isoperimetricriemannianpenroseinequality_2022}. In \cref{sec:admandco}, we analyze the relations with other notions of mass, most notably the classical $\ADM$ mass \cite{arnowitt_coordinateinvarianceenergyexpressions_1961}. This also gives us the occasion for discussing the physical relevance of these concepts. The other notions of mass that will be taken into account are the \emph{isocapacitary masses} \cite{jauregui_admmasscapacityvolumedeficit_2020, benatti_nonlinearisocapacitaryconceptsmass_2023}. Moreover, we show some (partly) new results about the relations among the connectedness of the isoperimetric sets of large volume and the equivalence of such notions of masses. We conclude this note with \cref{sec:questions}, where we present various possible directions of research. 

\subsection*{Acknowledgements.}
Part of this work has been carried out during the authors' attendance to the \emph{Thematic Program on Nonsmooth Riemannian and Lorentzian Geometry} that took place at the Fields Institute in Toronto. The authors warmly thank the staff, the organizers and the colleagues for the wonderful atmosphere and the excellent working conditions set up there. 
Luca Benatti is supported by the European Research Council’s (ERC) project n.853404 ERC VaReg -- \textit{Variational approach to the regularity of the free boundaries}, financed by the program Horizon 2020, by PRA\_2022\_11 and by PRA\_2022\_14. 
Mattia Fogagnolo has been supported by the European Union – NextGenerationEU and by the University of Padova under the 2021 STARS Grants@Unipd programme ``QuASAR". 
The authors are members of Gruppo Nazionale per l’Analisi Matematica, la Probabilit\`a e le loro Applicazioni (GNAMPA), which is part of the Istituto
Nazionale di Alta Matematica (INdAM), and are partially funded by the GNAMPA project ``Problemi al bordo e applicazioni geometriche".

The authors are grateful to G. Antonelli, S. Hirsch, L. Mazzieri, F. Oronzio and M. Pozzetta for countless, precious and pleasureful discussions on topics related to this work. Moreover, M. F. thanks A. Pluda, V. Franceschi and G. Saracco for having organized, with the support of INdAM, the wonderful workshop ``Anisotropic Isoperimetric Problems \& Related Topics”  in Rome.

The authors warmly Carlo Mantegazza for his thorough reading of the manuscript and for the many precious suggestions.
\section{The Inverse Mean Curvature Flow and the Hawking mass}
\label{sec:inverse}
A very powerful tool to understand some intimate geometric properties of noncompact $3$-manifolds with nonnegative scalar curvature is the flow of surfaces through the inverse of their mean curvature. We are focusing, for all the paper, on ambient manifolds with one single end. 
The Inverse Mean Curvature Flow (IMCF) of an immersion with strictly positive mean curvature $F_0 : \Sigma \hookrightarrow M$ is defined as 
\begin{equation}
\label{eq:smoothimcf}
    \frac{\partial}{\partial t} F_t = \frac{1}{\H_{t}} \nu_t,
\end{equation}
where $\nu_t$ is the normal pointing towards infinity (meaning towards the interior of the end, see \cref{fig:manifold}) and $\H_t$ the mean curvature, that is the sum of the principal curvatures. 
It is clear that, at least at those points where the mean curvature tends to zero, such flow must develop singularities\footnote{Singularities actually can happen in this case only \cite[Corollary 2.3]{huisken_higherregularityinversemean_2008}}. 

\begin{center}
\begin{figure}[ht]
\begin{tikzpicture}[x=1pt,y=1pt,scale=2.8]
\node[inner sep=0pt, line width=0pt, anchor=south west,scale=2.8,opacity=.9] at(-3.242,0){\includegraphics[width=100.845pt]{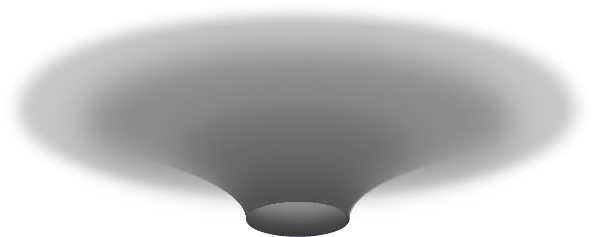}};
\draw[dashed] svg{M 36.05899,6.9381194 A 44.190375,15.027867 0 0 0 2.9467903,21.486099 44.190375,15.027867 0 0 0 47.137163,36.513967 44.190375,15.027867 0 0 0 91.327539,21.486099 v 0 A 44.190375,15.027867 0 0 0 58.215007,6.9380904};
\draw[dashed] svg{m 28.058864,11.119728 c -3.943844,1.442415 -8.439174,2.612121 -12.383499,3.4914};
\draw[dashed] svg{M 78.843409,14.654501 C 74.881968,13.777077 70.34095,12.603794 66.352884,11.152633};
\draw[dashed] svg{m 28.058864,11.119728 c -3.943844,1.442415 -8.439174,2.612121 -12.383499,3.4914};
\draw[dashed] svg {M 58.172965,6.9378924 A 43.849648,14.911995 0 0 0 47.180537,6.4617324 43.849648,14.911995 0 0 0 36.187781,6.9379214};
\draw svg{M 66.352884,11.152633 C 62.979478,9.9251344 60.001345,8.4989904 58.107911,6.8829924 m -21.917078,0.01911 c -1.877128,1.593502 -4.808986,3.002281 -8.131969,4.2176236};
\draw[dashed,line width=1pt] svg{m 31.863218,9.5477154 c 1.760601,-0.5415199 3.521089,-1.0830047 6.20281,-0.6917725 2.681721,0.3912322 6.284637,1.7151091 9.1725,1.6795001 2.887863,-0.03561 5.060366,-1.4305935 7.451997,-1.8401925 2.391631,-0.409599 5.002187,0.1661688 7.612745,0.7419369};
\draw[line width=1pt] svg{m 31.754003,9.5973034 c 3.880811,-1.820702 6.605884,-3.978031 6.629303,-6.449485 0.02341,1.648465 3.952976,2.984797 8.800337,2.984797 4.84736,0 8.776921,-1.336332 8.77695,-2.984797 2.9e-5,2.499519 2.787252,4.677677 6.738531,6.511305 1.689474,0.9419076 2.983032,1.6618366 1.517517,1.9472916 -1.465514,0.285453 -6.085777,-0.08552 -9.163318,0.529931 -3.07754,0.615447 -4.612037,2.217396 -8.608663,2.187938 -3.996626,-0.02946 -10.454888,-1.690399 -13.299321,-2.759177 -2.844431,-1.068782 -2.160963,-1.491431 -1.391336,-1.9678036 z};
\draw[line width=1pt] (47.18364,3.147841) ellipse (8.773778pt and 2.9837084pt);
\draw[->](58.296-3.242,11.944) --++(3,4) node[pos=.8, anchor= west]{$\nu$};
\node[anchor=north west] at (50.426-3.242,0) {$\partial M$};
\node at (85,37){$(M,g)$};
\draw[->] (35,5) to[out=45, in=180] (43,8);
\node [inner sep=1pt,anchor=north east] at (35,5){$\Omega$};
\node [inner sep=1pt,anchor=north east] at (46,19){$\partial \Omega$};
\end{tikzpicture}
\caption{}
\label{fig:manifold}
\end{figure}
\end{center}

To circumvent this fatal issue, Huisken-Ilmanen developed a weak notion of Inverse Mean Curvature Flow starting, for a bounded $\Omega$, at $\Sigma = \partial \Omega$, consisting in a proper function $w\in \Lip_{\loc}(M\smallsetminus \Int \Omega)$ weakly satisfying the boundary value problem 
\begin{equation}\label{eq:IMCF}
 \begin{cases}
 \div\left( \dfrac{ \D w}{\abs{\D w}}\right) &=& \abs{ \D w} & \text{on $M\smallsetminus \Omega$,}\\
 w &=& 0 & \text{on $\partial \Omega$,}\\
 w &\to& +\infty & \text{as $\dist(x,o) \to +\infty$},
 \end{cases}
\end{equation}
where $o \in M$ is any fixed point.
Observe that, if $w$ happens to be of class $\CS^2$ and free of critical points, the PDE in the above problem exactly states the velocity of the level sets $\{w= t\}$, that is $\abs{\D w}^{-1}$, coincides with the inverse of the mean curvature $\div ({ \D w}{\abs{\D w}{^{-1}}})$, hence one can just let $\Sigma_t = \{w =t\}$. On the other hand, Huisken-Ilmanen's definition makes sense also for $w$ that is only Lipschitz; in fact, it is defined as follows, for any $\Omega$ bounded closed set. For every $v \in \mathrm{Lip}_{\mathrm{loc}}(M)$ with $\{w \neq v\} \Subset M \setminus {\Omega}$ and any compact set $K\subset M \setminus \mathrm{Int}(\Omega)$ containing $\{w \neq v\}$,
\begin{equation}
\label{wimcf1}
J_w^K(w) \leq J_w^K(v)
\end{equation}
 where 
\begin{equation}
\label{J}
J_w^K(v) = \int_K \abs{\D v} + v \abs{\D w} \dd\mu. 
\end{equation}
The ``attainment" of the initial surface is codified by prescribing $w$ to be zero on $\Omega$. Observe that our choice of $\Omega$ to be closed allows to consider $\Omega = \partial M$.


Leaving the rigorous description of the properties of the weak IMCF to the original source \cite{huisken_inversemeancurvatureflow_2001}, we just heuristically illustrate how solutions behave. We shall distinguish between \emph{fat} level sets and non-fat ones. The latter consist in level sets $\set{w = t}$ that have not developed an interior of positive measure, and that in particular satisfy $\abs{\set{w = t}} = 0$, for the volume measure of $(M, g)$. As long as the level sets do not fatten, the weak flow consists in a foliation of $\CS^{1, \alpha}$-hypersurfaces moving { by the inverse of a suitable $L^2$-weak version of mean curvature}. In this case, in particular, the quantities $\abs{\partial \Omega_t}$, $\abs{\Omega_t}$ and  $\int_{\partial \Omega_t} \H^2$, are continuous, where we set $\Omega_t = \set{w \leq t}$. By contrast, if $\abs{\{w = \overline{t}\}} > 0$, we say that the level set is fat, and the time $\overline{t}$ will be referred to as a \emph{jump time}, as it can be described with a jump from $\set{w < \overline{t}}$ to its \emph{strictly outward minimizing hull}, which is defined as the set $E_t \supset \set{w < \overline{t}}$ of maximal volume among those minimizing the perimeter from the outside. In fact, it turns out that the closure of $E_t$ coincides with $\set{w \leq \overline{t}}$. It is not difficult to guess that jumps happen exactly when some $E_t$ is strictly enclosing $\{w < t\}$  with same perimeter. In particular, jump times are exactly those $t$'s such that $\abs{E_t} - \abs{\set{w < t}} > 0$ but $\abs{\partial E_t} = \abs{\partial \set{w < t}}$.

Besides the  work by Huisken-Ilmanen, we address the reader to
\cite{fogagnolo_minimisinghullscapacityisoperimetric_2022}, where a precise study of minimizing hulls (also in connection with IMCF) is performed taking advantage of the earlier \cite{bassanezi_subsolutionsleastareaproblem_1984}. Moreover, enlightening animations of the weak IMCF can be enjoyed at \cite{pasch_inversemeancurvatureflow_}.

\smallskip

A first decisive result of Huisken-Ilmanen's work \cite{huisken_inversemeancurvatureflow_2001} is a proof of the existence of a proper, unique, global weak flow anytime there exists a global subsolution. 
We are not describing this result in its full generality, but we just point out its application to $\CS^0$-asymptotically flat Riemannian $3$-manifold, together with some basic properties we are going to explicitly need. Before stating it, we give the precise definition of $\CS^k_\tau$-asymptotic flatness.
\begin{definition}[Asymptotically flat Riemannian manifolds]
\label{def:asyflat}
A Riemannian $3$-manifold $(M,g)$ with (possibly empty) boundary is said to be \emph{$\CS^{k}_\tau$-asymptotically flat}, with $k\in \N$ and $\tau > 0$ ($\tau = 0$ resp.),  if the following conditions are satisfied.
\begin{enumerate}
    \item There exists a compact set $K \subset M$ such that $M \smallsetminus K$ is diffeomorphic to $\R^3\smallsetminus \set{\abs{x}\leq R}$, through a map $(x^1,x^2,x^3)$ whose component are called \emph{asymptotically flat coordinates}.
    \item In the chart $(M \smallsetminus K, (x^1,x^2,x^3))$ the metric tensor is expressed as
    \begin{equation}
        g= g_{ij} \dd x^i \otimes\dd x^j= (\eum_{ij}+\eta_{ij}) \dd x^i \otimes\dd x^j
    \end{equation}
    with 
    \begin{align}
      \qquad \sum_{i,j=1}^{3}\sum_{\abs{\beta}=0}^k \abs{x}^{\abs{\beta}+\tau} \abs{ \partial_\beta \eta_{ij}} =O(1) \text{ ($=o(1)$ resp.)} && \text{as } \abs{x} \to +\infty.
    \end{align}
\end{enumerate}
We will denote the $\CS^k_0$-asymptotically flat condition simply with $\CS^k$-asymptotically flat. 
\end{definition}
Observe in particular that $\CS^0$-asymptotic flatness merely amounts to $\abs{\eta} \to 0$ as $\abs{x} \to +\infty$.

\smallskip

The following statement substantially gathers \cite[Proposition 3.2]{benatti_isoperimetricriemannianpenroseinequality_2022} and  \cite[Connectedness Lemma 4.2]{huisken_inversemeancurvatureflow_2001}, see also \cite[pp.9-10]{agostiniani_riemannianpenroseinequalitynonlinear_2022} and \cite[Lemma 2.1]{chan_monotonicitygreenfunctions_2022} for the connectedness part. 
\begin{theorem}[Existence and basic properties of the weak IMCF]
\label{thm:properties-imcf}
Let $(M, g)$ be a Riemannian $3$-manifold possibly with boundary. Suppose that $(M,g)$ is $\CS^0$-asymptotically flat. Then, for any closed $\Omega \supset \partial M$ with $\partial \Omega$ smooth, there exists a weak solution $w$ to  problem \cref{eq:IMCF}. 
If $\partial \Omega$ is connected, and $H_2(M, \partial M, \Z) =\{0\}$, then $\partial\{w \leq t\}$ is connected for any $t \in [0, +\infty)$.  
\end{theorem}

\subsection{The monotonicity of the Hawking mass}
\label{subsec:hawking}
Most of the results to be discussed in this paper are ultimately consequences of the monotonicity of the \emph{Hawking mass},  
\begin{equation}
\label{eq:hawkingmass}
\ma_{H}(\partial \Omega)= \frac{ \abs{ \partial \Omega}^{\frac{1}{2}}}{16 \pi^{\frac{3}{2}}} \left( 4 \pi - \int_{\partial \Omega}\frac{\H^2}{4} \dif \sigma \right)
\end{equation}
along $\Omega_t = \{w \leq t\}$. Such quantity has been substantially conceived in \cite{hawking_gravitationalradiationexpandinguniverse_1968}, while Geroch \cite{geroch_energyextraction_1973} showed it to be nondecreasing along any smooth IMCF of a connected surface in a $3$-manifold with nonnegative scalar curvature and devised it as a tool to provide the Positive Mass Theorem. 

Such computation is straightforward. It  
relies on classical evolution equations and on the Gauss-Bonnet Theorem, which is the reason why connectedness is needed along the evolution. Denote $\Sigma_t= F_t(\Sigma)$ and let $\dd \sigma_t$ be the area measure on $\Sigma_t$. Then, employing well-known evolution equations (see e.g. \cite[Theorem 3.2]{huisken_geometricevolutionequationshypersurfaces_1999}), we have 
\begin{equation}\label{eq:exponential_growth}
\frac{\dd}{\dd t}\abs{ \Sigma_t}= \int_{\Sigma_t} \frac{\partial}{\partial t}(\dd \sigma_t) =\int_\Sigma \dif\sigma_t = \abs{\Sigma_t},
\end{equation}
immediately implying $\abs{ F_t(\Sigma)}= \ee^t \abs{ \Sigma}$. Hence,  we get
\begin{align}
    (16 \pi) ^{\frac{3}{2}} \frac{\dd}{\dd t } \ma_H(\Sigma_t) &=\abs{ \Sigma_t}^{\frac{1}{2}}\left(8\pi-\int_{\Sigma_t}\frac{\H^2_t}{2} \dif \sigma_t -\frac{\dd}{\dd t } \int_{\Sigma_t}\H^2_t\dif \sigma_t\right)\\
    &=\abs{\Sigma_t}^{\frac{1}{2}}\left(8\pi-\int_{\Sigma_t}\frac{3\H^2_t}{2} \dif \sigma_t - 2\int_{\Sigma_t}\H_t \frac{\partial H_t}{\partial t}\dif \sigma_t\right)\\
    &=\abs{\Sigma_t}^{\frac{1}{2}}\left(8\pi-\int_{\Sigma_t}\frac{3\H^2_t}{2} \dif \sigma_t + 2\int_{\Sigma_t}\H_t \left(\Delta_{\Sigma_t} \frac{1}{ \H_t} + \frac{ \abs{ \h_t}^2+ \Ric(\nu_t, \nu_t)}{\H_t} \right)\dif \sigma_t\right),\intertext{where $\h_t$ is the second fundamental form of $\Sigma_t$.  Integrating by parts and using the classical (traced) Gauss-Codazzi equations, we obtain}
     (16 \pi) ^{\frac{3}{2}} \frac{\dd}{\dd t } \ma_H(\Sigma_t)&=\abs{\Sigma_t}^{\frac{1}{2}}\left(8\pi-\int_{\Sigma_t}\frac{3\H^2_t}{2} \dif \sigma_t + \int_{\Sigma_t}2\frac{ \abs{ \nabla_{\Sigma_t} \H_t}^2}{\H_t^2}+\sca - \sca^{\Sigma_t} + \abs{\h_t}^2 +\H_t^2\dif \sigma_t\right)\\
    &=\abs{\Sigma_t}^{\frac{1}{2}}\left(8\pi-\int_{\Sigma_t}\sca^{\Sigma_t} \dif \sigma_t + \int_{\Sigma_t}2\frac{ \abs{ \nabla_{\Sigma_t} \H_t}^2}{\H_t^2}+\sca+ \abs{\mathring{\h}_t}^2\dif \sigma_t\right),
\end{align}
where $\mathring{\h}_t$ is the trace-less second fundamental form of $\Sigma_t$. As $\sca\geq0$ and
\begin{equation}
\label{eq:gauss-bonnet}
    \int_{\Sigma_t}\sca^{\Sigma_t} \dif\sigma_t = 4\pi \chi(\Sigma_t) \leq 8\pi
\end{equation}
by the
Gauss-Bonnet Theorem applied to the connected surface $\Sigma_t$,
we have that the time derivative of the Hawking mass is nonnegative.

The analysis of the monotonicity in the weak formulation constitutes a central part of Huisken-Ilmanen's work. However, once accepted the heuristic description of the level set flow given above, one realizes that, at jump times, part of the evolving boundary should be replaced by a piece of minimal surface, hence the $L^2$-norm of the mean curvature  
can only decrease, while the area is continuous. Thus, the Hawking mass is expected only to increase. This is indeed what happens.
\begin{theorem}[Geroch monotonicity along the weak IMCF {\cite{huisken_inversemeancurvatureflow_2001}}]
\label{thm:geroch}
Let $(M, g)$ be a $\CS^0$-asympto\-tically flat Riemannian $3$-manifold possibly with boundary, satisfying $H_2(M, \partial M, \Z) = \{0\}$. Then, the Hawking mass \cref{eq:hawkingmass} is well defined and monotone nondecreasing along any weak solution of \cref{eq:IMCF} for any bounded closed $\Omega \supset \partial M$ with connected, smooth boundary $\partial \Omega$, as long as the flow is contained in a region of nonnegative scalar curvature.
\end{theorem}

\section{Isoperimetry in nonnegative scalar curvature}
\label{sec:isoperimetry}

\subsection{Nonnegative scalar curvature and the reverse isoperimetric inequality}
To see clearly the role played by the monotonicity of the Hawking mass in isoperimetric issues, we focus on the case of manifolds without boundary. It is not difficult to conceive that, shrinking and shrinking an initial geodesic ball $B(r, o)$ of radius $r$ centered at a desired $o \in M$, one can build through a limiting procedure as $r \to 0^+$ a weak IMCF originating from $o$. This will in turn be described by a function $w \in \mathrm{Lip}_{\mathrm{loc}}(M \setminus \{o\})$ solving in the weak sense \cref{eq:IMCF}, with the boundary condition replaced by $w(x)\to -\infty$ as $\mathrm{d}(x, o) \to 0^+$. This can actually be obtained with a similar argument as in \cite[Proposition 7.2]{huisken_inversemeancurvatureflow_2001}. In fact, arguing as in \cite[Blowdown Lemma 7.1]{huisken_inversemeancurvatureflow_2001}, $w(x)$ can be shown to asymptotically behave as the Euclidean model $(n-1) \log(\dist(x, o))$ around $o$. In particular, one immediately observes that $
\ma_H(\partial \Omega_t) \to 0$ when $\Omega_t= \set{w\leq t}$ is approaching $o$, that is when $t \to -\infty$. Consequently, being $\ma_H$ nondecreasing along the flow by \cref{thm:geroch}, recalling its expression \cref{eq:hawkingmass} the level sets of $w$ are seen to satisfy a \emph{reverse Willmore inequality}, meaning
\begin{equation}
\label{eq:reverse-willmore}
    \int_{\partial \Omega_t} \H^2 \dif\sigma \leq 16 \pi. 
\end{equation}
This suggests that the value of the Euclidean isoperimetric quotient of such sets are smaller than the value of the isoperimetric quotient of a round ball in flat $\R^3$. This can be obtained, in the smooth case, as follows.

Let $v(t) = \abs{\Omega_t}$. Then, by the coarea formula and equation \cref{eq:IMCF} solved by $w$, one has 
\begin{equation}
\label{eq:v'}
    v'(t) = \int_{\partial \Omega_t} \frac{1}{\H}\dif\sigma. 
\end{equation}
On the other hand, by  H\"older's inequality,
\begin{equation}
\label{eq:holder}
    \abs{\partial \Omega_t}=\int_{\partial \Omega_t} \H^\alpha \H^{-\alpha} \dif\sigma \leq \left(\int_{\partial \Omega_t} \H^{\alpha p}\dif\sigma\right)^{\frac{1}{p}} \left(\int_{\partial \Omega_t} \H^{-\alpha\frac{p}{(p-1)}} \dif\sigma\right)^{\frac{p-1}{p}},
\end{equation}
hence, letting $\alpha = 2/3$ and $p = 3$, one finds out that
\begin{equation}
\label{eq:stimashi}
    \left(\int_{\partial \Omega_t}\frac{1}{\H} \dif\sigma\right)^{-1} \leq \frac{\left(\int_{\partial \Omega_t} \H^2 \dif\sigma\right)^{\frac{1}{2}}}{\abs{\partial \Omega_t}^{\frac{3}{2}}}.
\end{equation}
Now, the numerator at the right-hand side is estimated by inequality \cref{eq:reverse-willmore}, while, as far as the denominator is concerned, we recall that the evolution forces the area $\abs{\partial \Omega_t}$ to be  $\mathrm{e}^t$ times a constant. This is due to formula \cref{eq:exponential_growth} (see \cite[Lemma 1.6]{huisken_inversemeancurvatureflow_2001} for the computation in the setting of weak solutions). By the asymptotic behaviour at the ``pole" pointed out above, the constant is in fact $4\pi$, that is the area of the round unit $2$-sphere, so
\begin{equation}
\label{eq:evolutionarea-pole}
   \abs{\partial \Omega_t} = 4\pi \mathrm{e}^t.  
\end{equation}
Consequently, combining these pieces of information with inequality \cref{eq:stimashi} and equation \cref{eq:v'} we conclude
\begin{equation}
\label{eq:v'ancora}
    v'(t) \geq 2\pi \mathrm{e}^{\frac{3}{2}t}.
\end{equation}
Integrating from $t \to -\infty$ (which corresponds to the pole) where both the volume and the area of the level sets vanishes, and taking into account again equality \cref{eq:evolutionarea-pole}, we get
\begin{equation}
\label{eq:reverse-iso}
\frac{\abs{\partial \Omega_t}}{\abs{\Omega_t}^{\frac{2}{3}}} \leq (36\pi)^{\frac{1}{3}} = \frac{\abs{\partial B}_{\R^3}}{\abs{B}_{\R^3}^{\frac{2}{3}}},
\end{equation}
where in the right-hand side $B$ denotes the round ball in $\R^3$.
In other words, the level sets of $w$ satisfy a reverse Euclidean isoperimetric inequality, at least when the evolution is smooth. 

The following result of Shi \cite{shi_isoperimetricinequalityasymptoticallyflat_2016}, deduced in particular by \cite[(15)]{shi_isoperimetricinequalityasymptoticallyflat_2016}, constitutes the general statement. Its less transparent formulation is ultimately due to the fact that not any volume is covered along the weak evolution, where jumps are allowed (and actually occur). 
\begin{theorem}[Shi's reverse Isoperimetric inequality]
\label{thm:shi}
Let $(M, g)$ be a $\CS^0$-asymptotically flat Riemannian $3$-manifold, satisfying $H_2(M, \partial M, \Z) = \{0\}$, and let $o \in M$. For $o \in M$, let $w_o$ be the weak IMCF issuing from $o$, and let, for $v > 0$
\begin{equation}
\label{eq:t(v)}
t(v) = \inf\set{\tau \, \st \, \abs{\set{w_o \leq t}} \geq v}.
\end{equation}
Then, as long as $(\{w_o \leq t(v)\}, g)$ has nonnegative scalar curvature, we have
\begin{equation}
    \label{eq:shi}
    \abs{\partial \{w_o \leq t(v)\}} \leq (36\pi)^{\frac{1}{3}} v^{\frac{2}{3}}. 
\end{equation}
\end{theorem}
The scalar curvature influences the existence of isoperimetric sets substantially by means of the above result. 

We find worth comparing the isoperimetric property of sublevel sets of the weak IMCF in nonnegative scalar curvature with those of the geodesic balls in nonnegative Ricci curvature.

\subsection{Comparison with nonnegative Ricci curvature}
\label{subsec:ricci}
    When the nonnegative scalar curvature assumption is strengthened to nonnegative Ricci curvature, we have that geodesic balls  satisfy the very same reverse isoperimetric inequality \cref{eq:reverse-iso}. This in turn relies on the basic Laplacian comparison and Bishop-Gromov theorems. Such classical results respectively state that, in a suitable weak sense, the mean curvature of a geodesic ball of radius $r$ is bounded from above by $(n-1)/r$, while its area $\abs{\partial B(r)}$ is controlled from above by $\abs{\S^{n-1}} r^{n-1}$. 
    
    We focus on dimension $3$, in order to facilitate the comparison with \cref{eq:reverse-iso}, although the argument works unchanged also in higher dimension. 
 Fix a point $o$ in a manifold $(M, g)$ endowed with a metric of nonnegative Ricci curvature $g$. 
 Letting $r(V)$ be the radius of the ball $B(r(V))$ centered in $o$ of volume $V$, it can be easily computed, through classical variation formulas, that
 \begin{equation}
\label{eq:derivative r(V)}
     r'(V) = \frac{1}{\abs{\partial B(r(V))}} \geq \frac{1}{4\pi r(V)^2},
 \end{equation}
 where the inequality is the Bishop-Gromov control on $\abs{\partial B(r)}$.
 Integrating such inequality, one gets
 \begin{equation}
\label{eq:boundr(V)}
     r(V) \geq \left(\frac{3}{4\pi} V\right)^{\frac{1}{3}}.
 \end{equation}
 On the other hand, we have
 \begin{equation}
\label{eq:derivativeI(V}
    \frac{\dd}{\dd V}  \abs{\partial B(r(V))} = \frac{1}{\abs{\partial B(r(V))}} \int_{\partial B(r(V))} \H \dif\sigma  \leq \frac{2}{r(V)} \leq \frac{2}{V^{\frac{1}{3}}} \left(\frac{4\pi}{3}\right)^{\frac{1}{3}},
 \end{equation}
 where the first inequality is due to the Laplacian comparison and the second one to \cref{eq:boundr(V)}. Integrating this other differential inequality, we conclude
 \begin{equation}
 \label{eq:reverse-iso-ricci}
\abs{\partial B(r)} \leq (36\pi)^{\frac{1}{3}} \abs{B(r)}^{\frac{2}{3}}
\end{equation}
for any radius $r \geq 0$. Observe that \eqref{eq:reverse-iso-ricci} fully mirrors \cref{eq:reverse-iso}, but for a different exhaustion. It should be clear that no particular features of the dimension three have been exploited here, contrarily to the Gauss-Bonnet Theorem \cref{eq:gauss-bonnet} utilized to infer the monotonicity of the Hawking mass, that in turn led to \cref{eq:reverse-iso}. 

{The reverse isoperimetric inequality in nonnegative Ricci curvature was first pointed out by Morgan-Johnson \cite{morgan_sharpisoperimetrictheoremsriemannian_2000}, and it was first exploited to infer the existence of isoperimetric sets under suitable asymptotic assumptions by Mondino-Nardulli \cite{mondino_existenceisoperimetricregionsnoncompact_2016}.}


\subsection{Isoperimetric analysis on manifolds with nonnegative scalar curvature}
The following is a version of Nardulli's generalized compactness principle \cite{nardulli_generalizedexistenceisoperimetricregions_2014}, crafted for an asymptotically flat framework. It substantially implies that in an isoperimetric minimizing sequence the portion of volume that is escaping at infinity is in fact converging to a ball in $\R^3$. The sublevel sets of the IMCF satisfying, as previously explained, a reverse isoperimetric inequality, can thus be exploited to replace such runaway volume and provide an isoperimetric set. In the remainder of this work, given an ambient manifold $M$, we denote with $I$ its isoperimetric profile, defined as the function $I: [0, \abs{M}) \to [0, \infty)$ given
\begin{equation}
\label{eq:profile}
    I(V) = \inf\set{P(E) \, \st \, E \subset M, \, \abs{E} = V}.
\end{equation}
Isoperimetric sets $E_V$ of volume $V$ are exactly those satisfying $\abs{\partial E_V} = I(V)$.
\begin{theorem}[Asymptotic decomposition of Isoperimetric minimizing sequences]
\label{thm:nardulli}
Let $(M, g)$ be a smooth, $\CS^0$-asymptotically flat Riemannian manifold $(M, g)$ with Ricci curvature bounded from below,  and let $I$ be its isoperimetric profile. Then, for any $V > 0$, there exists a possibly empty, \emph{bounded} $E \subset M$ and a ball $B$ in $\R^n$ such that $V = \abs{E} + \abs{B}$ and
\begin{equation}
    \label{eq:nardulli-isoperimetric}
    I(V) = \abs{\partial E} + \abs{\partial B}.
\end{equation}
\end{theorem}
The version above is actually a consequence of the way more general \cite[Theorem 1.1]{antonelli_isoperimetricproblemriemannianmanifolds_2021}.

Clearly, the number of the balls drifting away being at most one is due to the fact that a union of balls in $\R^3$ is manifestly isoperimetrically less convenient than a single one. 

\begin{remark}
Nardulli's earlier work has been recently vastly exploited and empowered in the context of possibly nonsmooth metric spaces with Ricci lower bounds (see \cite{antonelli_isoperimetricproblemriemannianmanifolds_2021, antonelli_existenceisoperimetricregionsmanifolds_2022, antonelli_sharpisoperimetriccomparisonnon_2022, antonelli_isoperimetricproblemdirectmethod_2022,  antonelli_isoperimetricproblemstructureinfinity_2023}). It led to far-reaching consequences regarding the existence of isoperimetric sets, isoperimetric inequalities and sharp properties of the isoperimetric profile in such setting. 
\end{remark}

We can now present the overall argument leading to the existence of isoperimetric sets. By \cref{thm:nardulli}, the only way a minimizing sequence for the isoperimetric problem at volume $V$ may fail to provide an isoperimetric set is by losing part of its volume at infinity, that, on the other hand, by the asymptotic flatness converges to a round ball $B$ in $\R^3$. However, the nonnegative scalar curvature, by \cref{thm:shi}, determines that a sublevel  set $F_B$ of a weak IMCF emanating from some point $o \in M$, enclosing the same volume as $B$, has no more perimeter than $B$. 
Thus, if such $B$ actually exists, than $E \cup F_B$ would provide an isoperimetric set of volume $V$ sitting in $M$. 

\begin{remark}
In the case of nonnegative Ricci curvature, a scheme like this has been worked out in \cite{mondino_existenceisoperimetricregionsnoncompact_2016}, taking advantage of  the isoperimetric properties of the geodesic balls described in \cref{subsec:ricci}. Variations on this argument have been proposed also in \cite{antonelli_isoperimetricproblemriemannianmanifolds_2021} and in \cite{antonelli_existenceisoperimetricregionsmanifolds_2022}.
\end{remark}

\smallskip

There is a last issue to be taken into account before safely running the above argument. As already pointed out, the weak IMCF can jump, and in particular there could be some value $V$ such that for no $t \in (-\infty, +\infty)$ the set $\{w \leq t\}$ encloses a volume exactly equal to $V$. This may cause trouble in case such volume is exactly the volume of the ball at infinity that we would like to replace. This problem is bypassed by the strict monotonicity of the isoperimetric profile, holding when $M$ is endowed with a minimal outermost boundary. 

\begin{proposition}
\label{prop:isomon}
Let $(M, g)$ be a $\CS^0$-asymptotically flat Riemannian $3$-manifold with Ricci curvature bounded from below and endowed with a closed, outermost minimal boundary.  Then, its isoperimetric profile $I$  is strictly increasing. 
\end{proposition}
\proof
We first of all recall that the isoperimetric profile is continuous in this case, see \cite{munozflores_localholdercontinuityisoperimetric_2019}.  Fix a volume $V$. Then, by \cref{thm:nardulli}, we know that
\begin{equation}
I (V) = \abs{\partial E} + \abs{\partial B}_{\R^3}
\end{equation}
for a possibly empty $E\subset M$ and for a Euclidean ball $B$ satisfying $V = \abs{E} + \abs{B}_{\R^3}$. Clearly, $E$ must be isoperimetric for its own volume, i.e. $\abs{\partial E} = I(\abs{E})$, and in particular it is smooth. Moreover, we recall that on a Riemannian manifold isoperimetric sets are bounded anytime an Euclidean-like isoperimetric inequality is available, at least for small volumes (see e.g. \cite[Theorem B.1]{antonelli_isoperimetricproblemriemannianmanifolds_2021}). In particular, this holds in our case: due to asymptotic flatness, an uniform Euclidean-like isoperimetric inequality can be checked to hold outside some compact set,  so \cite[Theorem 3.2]{pigola_connectivityinfinitymanifoldq_2014}  ensures its validity on the whole manifold with boundary. We can thus infer that $E$ is bounded.

We assume first that $E$ is nonempty.  We perform inward variations $E_t$ of $E$ obtained through a normal deformation with a velocity such that, at $t= 0$, is given by $\varphi \nu_{\partial E}$, for a nonpositive $\phi \in \CS^{\infty}_c (E \setminus \partial M)$, and where $\nu_{\partial E}$ points towards infinity. For a fixed $\varepsilon > 0$, we let $t(\varepsilon)$ be such that $\abs{E_{t(\varepsilon)}} = \abs{E} - \varepsilon$.  We can thus compute 
\begin{equation}
\label{eq:confriso}
\begin{split}
\liminf_{\varepsilon \to 0^+} \frac{I(V) - I(V - \varepsilon)}{\varepsilon} &\geq \liminf_{\varepsilon \to 0^+}\frac{\abs{\partial E}  - \abs{\partial E_{t(\varepsilon)}}}{\varepsilon}= t'(\varepsilon)_{|\varepsilon = 0} \int_{\partial E} \varphi \H_E \dif\sigma
\\ &=\frac{\int_{\partial E} \varphi \H_E \dif\sigma}{\int_{\partial E} \varphi \dif\sigma} = \H_E, 
\end{split}
\end{equation}
where $\H_E$ is the \emph{constant} mean curvature of $\partial E \setminus \partial M$.  In the first inequality in \cref{eq:confriso}, we used 
\begin{equation}
I(V-\varepsilon) \leq \abs{\partial E_{t(\varepsilon)}} + \abs{\partial B}_{\R^3},
\end{equation}
holding because the sets $E_\varepsilon \cup B_j$, with $B_j \subset M$ satisfying $\abs{B_j} = \abs{B}_{\R^3}$ approaching the round ball $B \subset \R^3$ in the $\CS^0$-topology as $j \to \infty$, form a valid family of competitors for the isoperimetric problem of volume $V - \varepsilon$. The monotonicity of $I$ thus follows by showing that the constant $\H_E$ is strictly positive. To see this, we can argue as in \cite[Lemma 2.8]{benatti_isoperimetricriemannianpenroseinequality_2022}, after \cite[Remark, p. 394]{huisken_inversemeancurvatureflow_2001}. Namely, one can flow a geodesic ball in the asymptotic region, that is mean-convex by \cite[Lemma 4.3]{benatti_asymptoticbehaviourcapacitarypotentials_2022}, through the Mean Curvature Flow of mean-convex surfaces with surgery in Riemannian $3$-manifolds \cite{brendle_meancurvatureflowsurgery_2018}; as proved in such paper, this flow smoothly converges to the outermost minimal boundary $\partial M$. One can then find a surface in the evolution ``touching" $\partial E \setminus \partial M$. If $\H_E$ were nonpositive, this would result in a contradiction with the maximum principle. This completes the proof in neighbourhoods of volumes $V$ such that $E$ as above is nonempty.  

In case $E$ were empty, deforming $B \subset \R^3$ inwardly as above yields \cref{eq:confriso}, this time in terms of the mean curvature of $B \subset \R^n$, that is obviously strictly positive. This concludes the proof.
\endproof

With \cref{thm:shi}, \cref{thm:nardulli} and \cref{prop:isomon}  at hand, we can prove that isoperimetric sets exist in any $\CS^0$-asymptotically flat Riemannian $3$-manifold with nonnegative scalar curvature and horizon boundary, for any volume, provided some lower bound on the Ricci curvature is in force. This is a refinement of \cite[Proposition K.1]{carlotto_effectiveversionspositivemass_2016}. Useful insights about the strategy employed were actually proposed by Brendle-Chodosh \cite{brendle_volumecomparisontheoremasymptotically_2014}, including the key computations leading to \cref{thm:shi}. 

\begin{theorem}[Existence of isoperimetric sets in nonnegative scalar curvature]
\label{thm:isoperimetric-existence}
Let $(M, g)$ be  $\CS^0$-asymptotically flat Riemannian $3$-manifold with nonnegative scalar curvature and with Ricci curvature bounded from below, endowed with a closed, minimal outermost boundary. Then, for any $V > 0$, there exists an isoperimetric set of volume $V$.
\end{theorem}
\proof
Let $V > 0$. As above, we have 
\begin{equation}
\label{eq:dec}
I (V) = \abs{\partial E} + \abs{\partial B}_{\R^3}
\end{equation}
for a possibly empty $E\subset M$ and for a round ball $B$ in flat $\R^3$ satisfying $V = \abs{E} + \abs{B}_{\R^3}$. We can assume that $B$ is nonempty, otherwise $E$ is already the isoperimetric set of volume $V$ sought for.  


\smallskip

Let now $o \in M$ be far away from the boundary. We can assume that there exists a weak IMCF $w_o$ issuing from $o$, although this is not obtained through flows of hypersurfaces homologous to $\partial M$, as \cref{thm:properties-imcf} would require. In fact, we can attach a cap to $\partial M$, and extend smoothly the metric $g$ to this new complete manifold without boundary \cite{pigola_smoothriemannianextensionproblem_2016}. Then, \cref{thm:properties-imcf}, coupled with the limiting procedure already mentioned above, yields a weak IMCF $w_o$ issuing from $o$. It is known that, in our assumptions, $H_2(M, \partial M, \Z) =\set{0}$, see e.g.  \cite[Lemma 2.8]{benatti_isoperimetricriemannianpenroseinequality_2022} and the references contained in the proof there, consequently, the manifold without boundary obtained attaching the cap and extending the metric satisfies $H_2(M, \Z) =\set{0}$. 

By \cite[Theorem 1.3]{mari_flowlaplaceapproximationnew_2022}, there exist functions $f_1 f_2,: [0, +\infty) \to \R$ growing to infinity at infinity, and such that
\begin{equation}
\label{eq:bound-mrs}
    f_1(\dist(x, o)) \leq w_o(x) \leq f_2 (\dist(x, o))  
\end{equation}
for any $o \in M$.
Such result requires a Ricci lower bound and the validity of a global, possibly weighted Sobolev inequality. In the $\CS^0$-asymptotically flat case this holds with no weight, as a direct consequence of the uniform Euclidean-like isoperimetric inequality that can be directly obtained outside some sufficiently large compact set,  then, by \cite[Theorem 3.2]{pigola_connectivityinfinitymanifoldq_2014}, on the whole manifold. 
\smallskip

In light of \cref{eq:bound-mrs}, given $v > 0$ we can always find $o$ sufficiently far in space such that $\{w_o \leq t(v)\}$ is disjoint from the bounded $E$ and $\partial M$. Since such sublevel set is contained in a nonnegative scalar curved region of a boundaryless Riemannian manifold satisfying $H_2(M, \Z) = \{0\}$, we can apply \cref{thm:shi} and infer that Shi's reverse isoperimetric inequality \cref{eq:shi} holds for $t(v)$. Choose now $v = \abs{B}$, and $o$ as above. Let $F = \set{w_o \leq t(v)}$, and observe that, by definition, $\abs{F} \geq v$. Now, if $\abs{F} = v$, then by \eqref{eq:shi} the set $E \cup F$ is isoperimetric of volume $V$. If instead $\abs{F} > v$, then by the strict monotonicity of $I$ shown in \cref{prop:isomon} and \eqref{eq:dec} we have
\begin{equation}
    \label{eq:chainiso1}
    I (\abs{E} + \abs{F}) > I(\abs{E} + v) =I(V) = \abs{\partial E} + (36\pi)^{\frac{1}{3}} v^{\frac{2}{3}}.
\end{equation}
On the other hand, we also have
\begin{equation}
\label{eq:chainiso2}
     I (\abs{E} + \abs{F}) \leq \abs{\partial E} +\abs{\partial F} \leq \abs{\partial E}  + (36\pi)^{\frac{1}{3}} v^{\frac{2}{3}}, 
\end{equation}
where we used again \eqref{eq:shi}. The chaining of \eqref{eq:chainiso1} with \eqref{eq:chainiso2} leads to a contradiction, that concludes the proof.  
 \endproof

 \section{The isoperimetric mass and the isoperimetric Penrose inequality}
 \label{sec:penrose}
So far, we did not mention any particular example  of $3$-manifold with nonnegative scalar curvature and minimal, outermost boundary. Let us focus now on the archetypal one, that will actually constitute the model for the geometric inequalities that we are going to present. The  Schwarzschild manifold of dimension $3$ of (positive) mass $\ma$ is the space $\R^3\smallsetminus \set{\abs{x} < 2 \ma}$ endowed with the rotationally symmetric metric
\begin{equation}
\label{eq:metric-schwar}
g = \left( 1+ \frac{ \ma}{2 \abs{x}} \right)^{4}\eum_{ij}\,\dd x^i \otimes \dd x^j.
\end{equation}
This Riemannian manifold has zero scalar curvature, the boundary $\partial M = \{{\abs{x} = 2\ma}\}$ is minimal, and, since any other level set of $\abs{x}$ has constant mean curvature, such boundary is also outermost, since the presence of any other closed minimal surface would result in a contradiction with the Maximum Principle. 

\smallskip

It is known from early work of Bray \cite{bray_penroseinequalitygeneralrelativity_1997}, later generalized in various directions \cite{bray_isoperimetriccomparisontheoremschwarzschild_2002, brendle_constantmeancurvaturesurfaces_2013}, that the isoperimetric sets in this warped products are only the Euclidean annuli given by $E_R = \{2\ma \leq \abs{x} \leq R\}$. Coherently with Shi's reverse isoperimetric inequality and with the fact that balls form an Inverse Mean Curvature Flow, one checks that
\begin{equation}
    \abs{E_R} -\frac{\abs{\partial E_R}^{\frac{3}{2}}}{6 \sqrt{\pi}} \geq 0.
\end{equation}
Moreover, the quantity at the left-hand side is seen to be asymptotic, as $R \to +\infty$, to $R^{2}$. Interestingly enough, one computes that the mass parameter $\ma$ is in fact recovered as the limit
 \begin{equation}
 \label{isomass-svarch}
 \ma = \limsup_{R \to +\infty} \frac{2}{\abs{\partial E_R}}\left(\abs{E_R} -\frac{\abs{\partial E_R}^{\frac{3}{2}}}{6 \sqrt{\pi}}\right).
\end{equation}
These observations  serve as a motivation for the notion of \emph{isoperimetric mass}, first introduced by Huisken \cite{huisken_isoperimetricconceptmassgeneral_2006}.
\begin{definition}[Isoperimetric mass]
Let $(M, g)$ be a Riemannian $3$-manifold possibly with boundary, and with infinite volume. Then, its \emph{isoperimetric mass} is defined as
    \begin{equation}
        \ma_{\iso}=\sup_{(\Omega_j)_{j\in \N}}\limsup_{j \to +\infty} \frac{2}{\abs{\partial \Omega}}\left( \abs{\Omega} -\frac{\abs{\partial \Omega}^{\frac{3}{2}}}{6 \sqrt{\pi}}\right),
    \end{equation}
    where the supremum is taken among all exhaustions $(\Omega_j)_{j\in \N}$ consisting of bounded domains with $\CS^{1,\alpha}$-boundary.
\end{definition}

It is immediately checked from the above definition that, if an exhaustion consists of isoperimetric sets, this automatically realizes the required supremum. In particular, if the scalar curvature is nonnegative, the following holds. 
\begin{lemma}
    \label{lem:isoperimetrici-mass}
  Let $(M, g)$ be a $3$-dimensional $\CS^0$-asymptotically flat Riemannian manifold with nonnegative scalar curvature and Ricci curvature bounded from below, having a closed, outermost minimal boundary. 
  Then, we have
  \begin{equation}
      \label{eq:isoperimetricimassa}
        \ma_{\iso} = \limsup_{V \to +\infty} \frac{2}{\abs{\partial E_V}}\left(\abs{E_V} -\frac{\abs{\partial E_V}^{\frac{3}{2}}}{6 \sqrt{\pi}}\right),
  \end{equation}
  where $E_V$ is an 
isoperimetric set  of volume $V > 0$.
\end{lemma}
It should be taken into account the following \emph{caveat}. The sequences of isoperimetric sets $(E_{V_j})_{j\in \N}$ of increasing volumes $V_j$ is not a priori form an exhaustion. However, the useful result \cite[Proposition 37]{jauregui_lowersemicontinuitymass0_2019} asserts that, in the $\CS^0$-asymptotically flat case,  one can relax the definition of $\ma_{\iso}$ in order to replace the requirement that $(\Omega_j)_{j\in \N}$ form an exhaustion, with the assumption $\abs{\partial \Omega_i} \to +\infty$ as $i \to +\infty$, establishing the validity of the above lemma.

\smallskip

Getting back to the Schwarzschild model, as a consequence of \cref{lem:isoperimetrici-mass}, we immediately deduce that $\ma_{\iso} = \ma$. Moreover, just by computing the value of $\abs{\partial M} = \{{\abs{x} = 2\ma}\}$ from the expression for $g$ in \cref{eq:metric-schwar}, one sees that
\begin{equation}
\label{eq:identity-mass-svar}
\sqrt{\frac{ \abs{\partial M}}{16 \pi}} = \ma_{\iso}.
\end{equation}
Together with Mazzieri, we proved that, in the general case considered in \cref{sec:isoperimetry} and actually with no Ricci lower bound required, the quantity $\sqrt{\abs{\partial M}/16\pi}$ is controlled from above by $\ma_{\mathrm{iso}}$. This kind of inequalities, bounding from below a suitable notion of mass in terms of the area of horizon boundary, are usually denominated Penrose inequalities.
\begin{theorem}[Isoperimetric Penrose inequality]
\label{iso-penrose}
Let $(M, g)$ be a $\CS^0$-asymptotically flat  Riemannian $3$-manifold with nonnegative scalar curvature and closed, minimal outermost \emph{connected} boundary. Then, 
\begin{equation}
\label{eq:iso-penrose}
\sqrt{\frac{ \abs{\partial M}}{16 \pi}} \leq  \ma_{\iso}.
    \end{equation}
    Moreover, equality holds if and only if $(M, g)$ is a Schwarzschild $3$-manifold with $\ma_{\iso} = \ma$.
\end{theorem}
Additional discussion on Penrose inequalities, including physical motivations, will be presented in the next section, where the isoperimetric mass will be compared with other concepts of mass, most notably with the ``classical" $\ADM$ mass (\cref{def:adm-mass}). For the time being, we just point out that it is well defined only under stronger assumptions on the asymptotic decay of the metric towards the flat one, and that these two notions of mass do actually coincide in such case. 

The general strategy for the proof of \cref{eq:iso-penrose} is the same as Huisken-Ilmanen's \cite{huisken_inversemeancurvatureflow_2001}, where the inequality was proved in terms of the $\ADM$ mass. It consists in exploiting again the monotonicity of the Hawking mass, this time along the weak Inverse Mean Curvature Flow of the minimal boundary. Since the mean curvature $\H$ vanishes on $\partial M$, the initial value of this quantity becomes
\begin{equation}
\label{eq:hawking-t0}
    \ma_H(\partial M) = \sqrt{\frac{\abs{\partial M}}{16\pi}},
\end{equation}
hence, since $\ma_H$ is nondecreasing along the flow by \cref{thm:geroch}, to prove \cref{eq:iso-penrose} it suffices to show that
\begin{equation}
\label{eq:asycomp}
    \lim_{t \to + \infty} \ma_H{(\partial \Omega_t)} \leq  \liminf_{t \to +\infty} \frac{2}{\abs{\partial \Omega_t}}\left( \abs{\Omega_t} -\frac{\abs{\partial \Omega_t}^{\frac{3}{2}}}{6 \sqrt{\pi}}\right) \leq \ma_{\mathrm{iso}}.
\end{equation}
This last step is completely different than Huisken-Ilmanen's asymptotic comparison with $\ma_{\ADM}$. Notably, it is by far computationally easier and does not require asymptotic flatness at all. 
\begin{remark}
In the proof of \cref{eq:iso-penrose} we just have to assume the existence of a proper weak IMCF's. The issue of the existence of such weak IMCF's is a very interesting topic \emph{per se}, related to some deep metric properties of the underlying manifold, see \cite{kotschwar_localgradientestimatesharmonic_2009, mari_flowlaplaceapproximationnew_2022}.
\end{remark}
We provide a (sketch of) the proof of the first asymptotic estimate appearing in \cref{eq:asycomp}  highlight the relation  with the computations leading to \cref{eq:stimashi} above.
\begin{proof}[Proof of \cref{eq:asycomp}]
We directly assume, for simplicity, that, as $t \to +\infty$, 
\begin{equation}
    \label{eq:h^2to16pi}
    \int_{\partial \Omega_t} \H^2 \dif\sigma \to 16\pi
\end{equation}
as $t \to +\infty$. We actually proved in \cite{benatti_isoperimetricriemannianpenroseinequality_2022} that this always can be assumed along the subsequence of $\Omega_t$'s we are going to consider, otherwise $\ma_{\mathrm{iso}}$ would be infinite, making \cref{eq:asycomp} trivial. 
We then can apply a version of de L'H\^opital rule (see e.g. \cite[Theorem A.1]{benatti_isoperimetricriemannianpenroseinequality_2022}), to get
\begin{equation}
\label{eq:hospital}
\liminf_{t \to +\infty} \frac{2}{\abs{\partial \Omega_t}}\left( \abs{\Omega_t} -\frac{\abs{\partial \Omega_t}^{\frac{3}{2}}}{6 \sqrt{\pi}}\right) \geq\liminf_{t \to +\infty} \frac{2}{\abs{\partial \Omega_t}}\left(\kern.1cm \int_{\partial \Omega_t} \frac{1}{\H} \dif \sigma -\frac{\abs{\partial \Omega_t}^{\frac{3}{2}}}{4 \sqrt{\pi}}\right).
\end{equation}
Applying the H\"older inequality as in \cref{eq:stimashi}  gives
\begin{equation}\label{eq:second_IMCF_de_lhopital}
\begin{split}
    \liminf_{t \to +\infty} \frac{2}{\abs{\partial \Omega_t}}\left( \abs{\Omega_t} -\frac{\abs{\partial \Omega_t}^{\frac{3}{2}}}{6 \sqrt{\pi}}\right) &\geq \liminf_{t \to +\infty}2 \left(\frac{\abs{ \partial \Omega_t}}{\int_{\partial \Omega_t}\H^2_t \dif \sigma}\right)^{\frac{1}{2}}\left( 1- \frac{1}{4\sqrt{\pi}}\left(\kern.1cm\int_{\partial \Omega_t} \H^2 \dif \sigma\right)^{\frac{1}{2}}\right)\\
    &=\liminf_{t \to +\infty}2 \left(\frac{\abs{ \partial \Omega_t}}{\int_{\partial \Omega_t}\H^2_t \dif \sigma}\right)^{\frac{1}{2}}\frac{ 1- \frac{1}{16\pi}\int_{\partial \Omega_t} \H^2 \dif \sigma}{1+( \frac{1}{16\pi}\int_{\partial \Omega_t} \H^2 \dif \sigma)^{\frac{1}{2}} }\\
    &=\liminf_{t \to +\infty}\frac{2 \ma_H(\partial \Omega_t)}{( \frac{1}{16\pi}\int_{\partial \Omega_t} \H^2 \dif \sigma)^{\frac{1}{2}}+\frac{1}{16\pi}\int_{\partial \Omega_t} \H^2 \dif \sigma} \, .
\end{split}
\end{equation}
The proof is completed by taking into account \cref{eq:h^2to16pi} and the monotonicity of $\ma_H$ along the weak IMCF, allowing us to show that limit inferior in the last line is actually a limit.
\end{proof}
Inequality \cref{eq:iso-penrose} states that  manifolds with nonnegative scalar curvature and minimal outermost boundary always tend to have more volume in big regions of  prescribed area than in the Schwarzchild model.
In the next section, through the notion of $\ma_{\mathrm{ADM}}$ mass, we will in particular show a direct link between  isoperimetric sets of large volumes in nonnegative scalar curvature and the energy/matter concepts in general relativity. 
\section{The \texorpdfstring{$\ADM$}{ADM} mass and comparison with other  masses}
\label{sec:admandco}
In order to better understand the heuristics behind the definition of the ${\ADM}$ mass below, we start by recalling why manifolds fulfilling the assumptions considered above are so natural in the context of General Relativity. This groundbreaking physical theory is formulated in the context of Lorentzian geometry and, specifically, in a $4$-manifold $(L, \gr)$ satisfying the \emph{vacuum} Einstein Field Equations
\begin{equation}
\label{eq:einsteneq}
    \Ric_\gr -\frac{1}{2}\sca_\gr = \mathrm{T},
\end{equation}
where $\mathrm{T}$ is called \emph{stress-energy} tensor and should be thought of as a physical datum. A largely accepted assumption on $\mathrm{T}$ is the \emph{dominant energy condition}, amounting simply to $\mathrm{T} (V, V) \geq 0$ for any \emph{timelike} vector $V$, that is $\gr(V, V) < 0$. Consider now a \emph{spacelike} $3$-hypersurface $(M, g)$, with $g$ induced by $\gr$, that is  $\gr(X, X) =  g(X, X) > 0$ for any $X \in T_pM$ and any $p \in M$. We also consider the simplest case of a \emph{time symmetric} $(M, g)$, consisting in the vanishing of the second fundamental form of the immersion $M \hookrightarrow L$. By means of the Gauss-Codazzi equations, \cref{eq:einsteneq}, coupled with the dominant energy condition $\mathrm{T} \geq 0$ induces on $(M, g)$ the identity
\begin{equation}
\label{eq:einstein-constrained}
    \sca_g = 16 \pi \rho \geq 0,
\end{equation}
where $\rho = \mathrm{T}(V, V)$ for some timelike vector $V$, hence the inequality follows from the dominant energy condition. We address the reader to \cite[Section 1.1]{carlotto_generalrelativisticconstraintequations_2021} for a detailed derivation of \cref{eq:einstein-constrained}.

The \emph{event horizon} of a \emph{black hole} in $(L, \gr)$ manifests itself in $(M, g)$ as a minimal outermost boundary. Finally, when modeling an \emph{isolated gravitational system}, it becomes natural to assume some kind of asymptotic flatness. With a physical language, this condition says that the gravitational field is not influenced by the presence of some mass ``at infinity".
We address the interested reader to \cite{hawking_largescalestructurespacetime_1973,lee_geometricrelativity_2019} for a comprehensive treatment of these relativistic concepts. 
\subsection{The \texorpdfstring{$\ADM$}{ADM} mass}
To get an idea of the $\ADM$ mass, we briefly stick with the more familiar Newtonian framework. 
Suppose to have a mass density $\rho$ in a $3$D-time snapshot represented by an isolated gravitational system $(M,g)$. Newton's formulation models the gravitational field by a function $V$ called \emph{gravitational potential}, which satisfies the relation $\Delta V= 4 \pi \rho$. Hence, using the divergence theorem, one can reconstruct the total mass of the system by observing the effect of the gravitational potential at infinity, since
\begin{equation}
    \ma = \int_M \rho \dif \mu = \int_M \frac{ \Delta V}{4\pi}\dif \mu = \lim_{r \to +\infty} \int_{\partial B_r} \frac{ \partial V}{ \partial r} \dif \sigma.
\end{equation}
On the other hand, we just recalled that, in Einstein's formulation, the gravitational field is modelled by a metric $g$ on $M$ which is only constrained (in the time-symmetric case) by the equation $\sca_g= 16 \pi \rho$. One then would like to compute the system's total mass simply by knowing the metric in this case as well. A first idea may be to integrate the quantity $\sca_g/16 \pi$ as we did in the Newtonian case. However, this approach has at least two main issues. The first one is that, even in a special paradigmatic case such as the Schwarzschild manifold, the total mass would be zero, which does not correspond to what we expect. The second problem is that the ``superposition principle" does not hold, since the scalar curvature is not a linear operator of $g$. In the case $(M,g)$ is close to the Euclidean flat space, one could try to circumvent this by replacing the scalar curvature with its first-order linearization, that is
\begin{equation}
    \sca_g \approx \sca_\delta + \D \sca\vert_{\delta}(g-\delta) =\frac{\dd }{\dd \varepsilon}\sca(\delta + \varepsilon(g-\delta) )\big\vert_{\varepsilon=0},
\end{equation}
since $\sca_\delta=0$. Observe that, in the above formula, $h=\delta + \varepsilon(g-\delta)$ is still a metric for $\abs{\varepsilon}$ small enough, 
 thus it makes sense to compute the scalar curvature of $h$, which is 
\begin{align}
    \sca_h &= h^{ij}(\partial_k H^k_{ij}- \partial_i H^k_{kj}+ H^k_{ij} H^{s}_{ks}- H^k_{is}H^s_{jk}) \\&=\frac{\varepsilon}{2}\delta^{ij}\delta^{ks}\partial_k(-\partial_s g_{ij}+ \partial_j g_{is} + \partial_i g_{js})-\frac{\varepsilon}{2}\delta^{ij}\delta^{ks}\partial_i(-\partial_s g_{kj}+ \partial_j g_{ks} + \partial_k g_{js})+ O_1(\varepsilon^2) \\
    &=\frac{\varepsilon}{2}\delta^{ij}\delta^{ks}\partial_k(-\partial_s g_{ij}+ 2\partial_j g_{is})-\frac{\varepsilon}{2}\delta^{ij}\delta^{ks}\partial_i \partial_j g_{ks} + O_1(\varepsilon^2)\\
    &= \varepsilon \delta^{ij}\delta^{ks} (\partial_k \partial_i g_{js} - \partial_s \partial_k g_{ij}) +O_1(\varepsilon^2)
\end{align}
where $H$ are the Christoffel symbols of $h$. We are employing the normal coordinates of flat $\R^3$. Hence,
\begin{equation}\label{eq:approximation_scalar_curvature}
    \sca_g\approx\delta^{ij}\delta^{ks} (\partial_k \partial_i g_{js} - \partial_k \partial_s g_{ij}).
\end{equation}
By means of the divergence theorem, one has
\begin{align}
\label{eq:conti-adm}
    \frac{1}{16 \pi }\int_M \sca_g \dif \mu_g &\approx\frac{1}{16 \pi }\int_M \delta^{ij}\delta^{ks} (\partial_k \partial_i g_{js} - \partial_k \partial_s g_{ij})\dif \mu_\delta \\&= \lim_{r \to +\infty} \frac{1}{16 \pi }\int_{\partial B_r}\delta^{ij}( \partial_i g_{jk} - \partial_k g_{ij})\frac{x^k}{\abs{x}}\dif \sigma_\delta\\
    &= \lim_{r \to +\infty} \frac{1}{16 \pi }\int_{\partial B_r}g^{ij}( \partial_i g_{jk} - \partial_k g_{ij})\nu^k\dif \sigma_g.
\end{align}
The quantity appearing on the right-hand side is called Arnowitt-Deser-Misner's mass $\ma_{\ADM}$ \cite{arnowitt_coordinateinvarianceenergyexpressions_1961}. 
If the metric $g$ is sufficiently close to the flat metric, the above computation shows that the integral of the scalar curvature is approximated by such mass. Clearly, the two quantities are not, in general, the same, since the first one depends on the global behaviour of the metric, while the second one depends only on its behaviour at large distances. On the other hand, it has been showed in \cite{arnowitt_coordinateinvarianceenergyexpressions_1961} that, in a suitable asymptotic setting, the finiteness of the one is equivalent implies the finiteness of the other. To see this, one can consider the vector field,
\begin{equation}
    Y^k = g^{kl}g^{ij}(\partial_i g_{jl} - \partial_l g_{ij}) = g^{ij} \Gamma^{k}_{ij} -g^{ik}\Gamma^{j}_{ij},
\end{equation}
inspired by the fact that, formally, $\ma_{\ADM} = \lim_{r\to + \infty} \int_{\partial B_r} \langle Y, \nu \rangle \dif \sigma_g$. This suggests to compute the divergence of $Y$ with respect to $g$, that is
\begin{equation}
\label{eq:diveY}
\begin{split}
   \mathrm{div} Y &=  \partial_k Y^k + \Gamma^l_{lk}Y^k \\
   &= \partial_k g^{ij}\Gamma^k_{ij} + g^{ij}\partial_k \Gamma^k_{ij}-\partial_k g^{ik} \Gamma^{j}_{ij} -  g^{ik}\partial_k\Gamma^{j}_{ij} +g^{ij}\Gamma^l_{lk}\Gamma^{k}_{ij}-g^{ik}\Gamma^l_{lk}\Gamma^{j}_{ij}\\&= \sca_g +\partial_k g^{ij}\Gamma^k_{ij}-\partial_k g^{ik} \Gamma^{j}_{ij}.
\end{split}
\end{equation}
The asymptotic conditions we assume  are
\begin{align}
\label{eq::condadm}
    \kst^{-1}\delta \leq g \leq \kst \delta  & \text{ in $M\smallsetminus B_R$}, & \int_{M \smallsetminus B_R} \abs{ \partial g}^2 \dif \mu_\delta<+\infty
\end{align}
and we show that the limit defining the $\ADM$ mass exists. Appealing  to the divergence theorem and to \cref{eq:diveY}, we have
\begin{align}
\label{eq:relYscaletc}
    \abs{\int_{\partial B_s} \ip{Y|\nu} \dif \sigma_g - \int_{\partial B_r} \ip{Y |\nu} \dif \sigma_g - \int_{B_s \smallsetminus B_r} \sca_g \dif \mu_g}  \leq \kst \int_{B_s \smallsetminus B_r}  \abs{ \partial g}^2 \dif \mu_\delta
\end{align}
for every $R\leq r<s$. 

Assume now that $\sca_g^{-}=-\min\set{\sca_g,0}\in L^1(M)$.
Taking the inferior limit in \cref{eq:relYscaletc},  as $s \to +\infty$, we get
\begin{equation}
    \liminf_{s \to +\infty}\int_{\partial B_s} \ip{Y|\nu} \dif \sigma_g \geq -\int_{M\smallsetminus B_r} \sca^-_g \dif \mu_g -   \mathrm{C}\int_{M\smallsetminus B_r} \abs{ \partial g}^2 \dif \mu_\delta +\int_{\partial B_r} \ip{Y |\nu} \dif \sigma_g,
\end{equation}
and considering the superior limit as $r \to +\infty$, in view of our assumptions we get
\begin{equation}
    \liminf_{s \to +\infty}\int_{\partial B_s} \ip{Y|\nu} \dif \sigma_g \geq \limsup_{r \to +\infty}\int_{\partial B_r} \ip{Y |\nu} \dif \sigma_g,
\end{equation}
proving the existence of the limit defining $\ma_{\ADM}$. Moreover, since 
\begin{equation}
    \abs{\ma_{\ADM}-\int_{M\smallsetminus B_R} \sca_g \dif \mu_{g}} \leq\int_{\partial B_R} \abs{Y}\dif \sigma_g + \kst \int_{M \smallsetminus B_R} \abs{ \partial g}^2 \dif \mu_\delta
\end{equation}
the $\ADM$ mass is finite if and only the scalar curvature is integrable on $M$.

In order to complete the presentation of the $\ADM$ mass, we should ensure its independence on the asymptotic chart at infinity chosen, so that the rightmost-hand side of \cref{eq:conti-adm} provides a well-posed definition. This has been proved under slightly stronger asymptotic conditions than those in \cref{eq::condadm} by Bartnik \cite{bartnik_massasymptoticallyflatmanifold_1986} and Chru{\'s}ciel \cite{chrusciel_boundaryconditionsspatialinfinity_1986}, that is assuming $g$ to be $\CS^{1}_\tau$-asymptotically flat, $\tau>1/2$.  
\begin{definition}
\label{def:adm-mass}
    Let $(M, g)$ be a $\CS^{1, \tau}$-asymptotically flat Riemannian $3$-manifold, for $\tau > 1/2$, and such that $\sca_g^- \in L^1(M)$. Then, the $\ma_{\ADM}$ mass is  (well) defined as
    \begin{equation}
\label{eq:madm}
        \ma_{\ADM} = \lim_{r \to +\infty} \frac{1}{16 \pi }\int_{\partial B_r}g^{ij}( \partial_i g_{jk} - \partial_k g_{ij})\nu^k\dif \sigma_g.
    \end{equation}
    Moreover, it is finite if and only if the scalar curvature is integrable on $M$.
\end{definition}
It has actually been shown \cite{chrusciel_boundaryconditionsspatialinfinity_1986} that the threshold $\tau > 1/2$ is sharp, since one can exhibit $\CS^1_{1/2}$-asymptotically flat metrics such that the limit in \cref{eq:madm} gives arbitrary values according to the selected chart at infinity. 
\subsection{Equivalence between masses}
A large amount of literature deals with geometric inequalities involving the $\ADM$ mass. Most notably, Schoen-Yau's Positive Mass Theorem \cite{schoen_proofpositivemassconjecture_1979, schoen_proofpositivemasstheorem_1981}, ensures, under suitable decay conditions on the metric, that $\ma_{\ADM} \geq 0$ in nonnegatively scalar curved complete manifolds, and that it vanishes only on flat $\R^3$. Solving a special case of a conjecture of Penrose \cite{penrose_nakedsingularities_1973}, Huisken-Ilmanen \cite{huisken_inversemeancurvatureflow_2001} sharpened such result in the connected  boundary case, obtaining the (Riemannian) Penrose inequality
\begin{equation}
\label{eq:adm-penrose}
    \sqrt{\frac{ \abs{\partial M}}{16 \pi}} \leq  \ma_{\ADM}.
\end{equation}
With the description given above of $\ma_{\ADM}$ as an actual (candidate) for a  global physical mass, the validity of the (Riemannian) Penrose inequality should appear more natural. In fact, the Schwarzschild metric, where the inequality is checked to hold as an equality, represents a space where the whole matter is shielded by the horizon $\partial M$. Consequently, any other space should have a larger global mass, in relation with the area of the boundary, as \cref{eq:adm-penrose} actually states.

Huisken-Ilmanen's proof of \cref{eq:adm-penrose} exploits their \cref{thm:geroch} coupled with an involved asymptotic analysis resulting in $\lim_{t \to +\infty}\ma_H(\partial \Omega_t) \leq \ma_{\ADM}$, along the weak IMCF of the horizon boundary. Bray \cite{bray_proofriemannianpenroseinequality_2001}, with a completely different proof based on an \emph{ad hoc} developed conformal flow of metrics, removed the assumption of connected boundary. In both approaches, the decay assumptions are strictly stronger than in \cref{def:adm-mass} t namely $\CS^{1}_1$-asymptotically flatness and $\Ric \geq -\kst \abs{x}^{-2}$ in \cite{huisken_inversemeancurvatureflow_2001}, and $\CS^{2}_1$-asymptotically flatness in \cite{bray_proofriemannianpenroseinequality_2001}.

\smallskip

The question about the validity of \cref{eq:adm-penrose} in the optimal $\CS^1_{\tau}$-asymptotically flat regime arises then naturally, and has been answered in collaboration with Mazzieri \cite[Theorem 1.1]{benatti_isoperimetricriemannianpenroseinequality_2022}.
Another even more natural (and more general) question regards the relation between $\ma_{\iso}$ and $\ma_{\ADM}$. Sharpening the earlier result envisioned by Huisken \cite{huisken_isoperimetricconceptmassgeneral_2006} and rigorously obtained (and strengthened) by Jauregui-Lee \cite{jauregui_lowersemicontinuitymass0_2019}, we do obtain also \cite[Theorem 4.13]{benatti_isoperimetricriemannianpenroseinequality_2022} the stronger conclusion
\begin{equation}
\label{eq:identity-masses}
\ma_{\iso} = \ma_{\ADM}
\end{equation}
in  $\CS^{1}_\tau$-asymptotically flat Riemannian manifolds with nonnegative scalar curvature and horizon boundary, for $\tau > 1/2$. We observe that such identity provides also a  proof of the well-posedness of $\ma_{\ADM}$ in this asymptotic regime completely different from the ones of Bartnik and Chr\`usciel \cite{bartnik_massasymptoticallyflatmanifold_1986, chrusciel_boundaryconditionsspatialinfinity_1986}, in the case of nonnegative scalar curvature. Indeed, the isoperimetric mass is manifestly independent on the choice of the coordinate chart.

The proofs of both the Penrose inequality \cref{eq:adm-penrose} and of \cref{eq:identity-masses} under optimal decay assumptions in \cite{benatti_isoperimetricriemannianpenroseinequality_2022} are based on some new insights about the asymptotic behaviour of the weak IMCF. The sharpening of the decay assumptions in the Penrose inequality substantially stems from the asymptotic comparison \cref{eq:asycomp} and on its relation with the asymptotic behaviour of the harmonic functions that takes advantage of the potential-theoretic version of the Hawking mass introduced in \cite{agostiniani_greenfunctionproofpositive_2021}. The equality between the two masses \cref{eq:identity-masses} fundamentally relies also on a combination with Huisken-Jauregui-Lee argument \cite{huisken_isoperimetricconceptmassgeneral_2006, jauregui_lowersemicontinuitymass0_2019} building on the Mean Curvature Flow.

\smallskip

In what follows, we
 outline an alternative strategy devised by Chodosh-Eichmair-Shi-Yu \cite{chodosh_isoperimetryscalarcurvaturemass_2021} that can be useful to prove \cref{eq:identity-masses}, showing a relation with the isoperimetric sets that have been discussed before.

\smallskip

Before going on, we point out that the inequality 
\begin{equation}
\label{eq:madmleqmiso}
    \ma_{\ADM} \leq \ma_{\iso}
\end{equation}
can be obtained in the optimal asymptotic regime through a direct yet not trivial computation that has been carried out in \cite{fan_largespheresmallspherelimitsbrownyork_2009},  so we will deal with the reverse one only. 
 A key observation is the following,  deduced by the proof of \cite[Theorem C.1]{chodosh_isoperimetryscalarcurvaturemass_2021}. When isoperimetric sets are known to exist, as in the setting of \cref{thm:isoperimetric-existence}, they do realize, in the limit of infinite volume, by \cref{lem:isoperimetrici-mass},  the isoperimetric mass. Since each of them has obviously constant mean curvature, the asymptotic comparison argument carried out in \cite{benatti_isoperimetricriemannianpenroseinequality_2022} can be reversed in some sense, allowing to estimate the isoperimetric mass \emph{from above} with the Hawking mass. This inspiring idea is expressed in the computation  \cref{eq:Cholimsupestimate}.

Starting from the next proposition and for all results up to the end of \cref{sec:admandco}, we will drop the outermost condition on the boundary.
This is due to known properties of asymptotically flat $3$-manifolds of nonnegative scalar curvature, see \cite[Lemma 2.8]{benatti_isoperimetricriemannianpenroseinequality_2022}, asserting, in particular, that if the manifold contains minimal surfaces, then one can find a minimal, outermost $\Sigma$ enclosing $\partial M$ and the analysis is applied on the new manifold with the minimal outermost boundary $\Sigma$. The masses are obviously the same as those of the original manifold. If $(M, g)$ possesses no closed minimal surfaces, and, in particular, it is boundaryless, then one can safely work in such complete manifold.

\begin{proposition}\label{prop:misomadm}
    Let $(M,g)$ be a complete $\CS^1_{\tau}$-asymptotically flat Riemannian $3$-manifold, for $\tau>1/2$, with nonnegative scalar curvature and (possibly empty) closed, minimal boundary. Assume that there exists $V_0 > 0$ such that, for any $V \geq V_0$, there exists an isoperimetric set $E_V$ of volume $V$ with connected boundary. Then,
\begin{equation}\label{eq:misomadm}
    \ma_{\iso} = \ma_{\ADM}.
\end{equation}
\end{proposition}

\begin{proof}
As already pointed out, we only have to show $\ma_{\iso} \leq \ma_{\ADM}$. 
    Assume $\ma_{\ADM}<+\infty$ otherwise we are done, and $\ma_{\iso}>0$, otherwise \cite[Theorem 1.3]{benatti_isoperimetricriemannianpenroseinequality_2022} implies that $(M,g)$ is isometric to $\R^3$ and the result trivially holds. 

    Let $I$ be the isoperimetric profile of $(M, g)$, and $E_V$ an isoperimetric set of volume $V$, meaning that $I(V) = \abs{\partial E_V}$. Applying as above a version of de L'H\^opital rule, we get
    \begin{align}\label{eq:Cholimsupestimate}
\begin{split}
    \ma_{\iso}&= \limsup_{V \to +\infty} \ma_{\iso}(E_V) = \limsup_{V\to +\infty} \frac{2}{I(V)} \left(V - \frac{I(V)^\frac{3}{2}}{6 \sqrt{\pi}}  \right)\\&\leq \limsup_{V\to +\infty}\frac{2}{I'(V)}\left( 1 - \frac{\sqrt{I(V)} I'(V)}{4 \sqrt{\pi}} \right)\\
    & = \limsup_{V \to +\infty}\frac{ 2\sqrt{I(V)}\left(1 -\frac{I(V) I'(V)^2}{16 \pi}  \right)}{\sqrt{I(V)}I'(V)\left(1 + \frac{\sqrt{I(V)}I'(V)}{4\sqrt{\pi}}\right)}\\&= \limsup_{V\to +\infty} \frac{32 \pi\, \ma_{H}(\partial E_V)}{4 \sqrt{\pi}I'(V)\sqrt{I(V)} + I'(V)^2 I(V)},
    \end{split}
\end{align}
where $I(V)= \abs{ \partial E_V}$. 
If the Hawking mass of the isoperimetric sets of large volumes satisfies the sharp bound $\ma_H(\partial E_V) \leq \ma_{\ADM}$, \cref{eq:Cholimsupestimate} allows to conclude with the desired bound $\ma_{\iso} \leq \ma_{\ADM}$. 
    
    Let then $(V_k)_{k \in \N}$ be a sequence realising the superior limit in \cref{eq:Cholimsupestimate}. Since $\ma_{\iso}>0$, $\ma_H(\partial E_{V_k})\geq 0$ for large $k$. In particular, since $I'(V_k) = \mathrm{H}_k$ the constant value of the mean curvature of $\partial E_{V_k}$, we have
    $I'(V_{k})^2 I(V_{k})\leq 16 \pi$. Assume now by contradiction that there exists a (not relabeled) subsequence such that $I'(V_{k})^2 I(V_{k}) \leq 16 \pi -\varepsilon$ for some positive $\varepsilon>0$. Then, we would have
    \begin{equation}
    \label{eq:bound1}
        \frac{\varepsilon}{32 \pi} \sqrt{I(V_k)} \leq \ma_{H}(\partial E_{V_k}) \leq \ma_{\ADM},
    \end{equation}
    where the last inequality follows as in \cite[Theorem 1.1]{benatti_isoperimetricriemannianpenroseinequality_2022}, giving a contradiction. In particular, $I'(V_k)^2I(V_k)\to 16 \pi $ as $k \to +\infty$. Plugging this piece of information into \cref{eq:Cholimsupestimate}, we get 
    \begin{equation}
    \label{eq:bound2}
        \ma_{\iso} \leq \limsup_{k \to +\infty} \ma_{H}(\partial E_{V_k} ) \leq \ma_{\ADM},
    \end{equation}
    proving \cref{eq:misomadm}. We stress the fact that the connectedness of $\partial E_{V_k}$ is required in order to infer the rightmost bound in \cref{eq:bound1} and in \cref{eq:bound2}.
\end{proof}
The authors \cite{chodosh_isoperimetryscalarcurvaturemass_2021} can, indeed, count on connectedness, as they are working in  $\CS^2_1$-asymptotically flat manifolds, where isoperimetric sets of large volume are close to coordinate balls, as they show in \cite[Theorem 1.1]{chodosh_isoperimetryscalarcurvaturemass_2021} (actually valid in $\CS^2_\tau$-asymptotically flat manifolds with $\tau > 1/2$), that  importantly weakens the assumptions of the earlier works of Eichmair-Metzger \cite{eichmair_largeisoperimetricsurfacesinitial_2013, eichmair_uniqueisoperimetricfoliationsasymptotically_2013}, where the manifolds were assumed to be asymptotic to Schwarzschild manifolds. Moreover, in \cite{chodosh_isoperimetryscalarcurvaturemass_2021}, the authors  also rely on Huisken-Ilmanen's bound on the Hawking mass,  thus their analysis could be pushed to $\CS^2_\tau$-asymptotically flat manifolds, for $\tau > 1/2$, coupled with $\Ric \geq -\kst \abs{x}^{-2}$. 
Through the $\ADM$-Penrose inequality in optimal decay assumptions \cite[Theorem 1.1]{benatti_isoperimetricriemannianpenroseinequality_2022} the condition on the Ricci curvature can be dropped.
Actually, in \cite{chodosh_isoperimetryscalarcurvaturemass_2021}, the authors also take advantage of an \emph{a priori} knowledge of $I(V)$ and $I'(V)$ as $V \to +\infty$. The proof of \cref{prop:misomadm} shows that this is not needed.

\subsection{Nonlinear masses}
In this last section, we briefly introduce and discuss the nonlinear potential theoretic counterpart of the isoperimetric mass. To this end, we first introduce the following $p$-capacity of a compact set $K \in M$,
\begin{equation}
\label{p-cap}
\ncapa_p(K) = \inf \set{\frac{1}{4\pi}\left(\frac{p-1}{3-p}\right)^{p-1}\int_{M\smallsetminus K} \abs{\D v}^p \dif \mu\st v \in \CS_c^\infty(M),\, v \geq 1\text{ on } K}.
\end{equation}
As a consequence of the remarkable \cite[Theorem 3.6]{mari_flowlaplaceapproximationnew_2022}, $\CS^0$-asymptotically flat manifolds are $p$-nonparabolic, a condition consisting in the existence of a positive $p$-harmonic Green's function that in particular implies that the above quantity is positive for any $K = \partial \Omega$ of class $\CS^{1, \alpha}$, if $1 < p < 3$. Similarly to the isoperimetric mass, one can define the \emph{$p$-isocapacitary mass} as follows. 

\begin{definition}
Let $(M, g)$ be a Riemannian $3$-manifold possibly with boundary, with infinite volume. Then, its $p$-\emph{isocapacitary mass} is defined as
    \begin{equation}
    \label{eq:pisomass}
        \ma^{\tp}_{\iso}= \sup_{(\Omega_j)_{j \in \N}} \limsup_{j \to +\infty} \frac{1}{2p \pi \ncapa_p(\partial \Omega_j)^{\frac{2}{3-p}}}\left( \abs{\Omega_j} -\frac{4 \pi}{3} \ncapa_p (\partial \Omega_j)^\frac{3}{3-p}\right).
    \end{equation}
    where the supremum is taken among all exhaustions $(\Omega_j)_{j\in \N}$ consisting of domains with $\CS^{1,\alpha}$-boundary.
\end{definition}
The above capacitary notion of mass has recently been considered for $p = 2$ by Jauregui \cite{jauregui_admmasscapacityvolumedeficit_2020}, and in \cite{benatti_nonlinearisocapacitaryconceptsmass_2023} for $1 < p <3$. After the discussion in the last sections, one could naturally wonder about the relation of $\ma^{\tp}_{\iso}$ with the isoperimetric mass and the  ${\ADM}$ mass. A first answer in this direction is the following result, holding in the generality of $\CS^0$-asymptotically flat $3$-manifolds. 
\begin{proposition}[{\cite[Proposition 5.6]{benatti_nonlinearisocapacitaryconceptsmass_2023}}]
\label{prop:mp<miso}
Let $(M, g)$ be a $\CS^0$-asymptotically flat Riemannian $3$-manifold. Then,
\begin{equation}
\label{eq:mp<miso}
    \ma^{\tp}_{\iso} \leq \ma_{\iso}. 
\end{equation}
\end{proposition}
Its proof is classical in nature and can be compared to the derivation of isocapacitary inequalities from the isoperimetric one (see \cite{mazya_sobolevspaces_2013} or, more directly, to the arguments in \cite{jauregui_capacityvolumeinequalitypoincare_2012} and in the proof of \cite[Theorem 4.1]{benatti_minkowskiinequalitycompleteriemannian_2022}). More precisely, it is based on the  P\'olya-Szeg\"o inequality applied to the sharp asymptotic isoperimetric inequality
\begin{equation}
        (6 \sqrt{\pi} \abs{\Omega})^{\frac{2p}{3}} \leq \abs{\partial \Omega}^{p}+ 2p\sqrt{\pi}(\ma_{\iso}+o(1))\abs{ \partial \Omega}^{\frac{2p-1}{2}} \qquad \left(\text{for }\ma_{\iso}>-\infty\right),
\end{equation}
that is actually a direct consequence of the definition of the isoperimetric mass (see \cite[Theorem 5.5]{benatti_nonlinearisocapacitaryconceptsmass_2023} for more details). 

The reverse inequality is obtained in \cite{benatti_nonlinearisocapacitaryconceptsmass_2023} in the setting of $\CS^1_\tau$-asymptotically flat Riemannian $3$-manifolds with $\tau > 1/2$, of nonnegative scalar curvature. As these assumptions may suggest, we strongly rely on the well-posedness of the $\ADM$ mass, appearing in an asymptotic computation along geodesic spheres, as in \cite{fan_largespheresmallspherelimitsbrownyork_2009}. The overall argument consists in a strengthening of the one proposed in the proof \cite[Theorem 5]{jauregui_admmasscapacityvolumedeficit_2020} and in the application of a generalization of a capacitary estimate of Bray-Miao \cite{bray_capacitysurfacesmanifoldsnonnegative_2008}, due to Xiao \cite{xiao_harmoniccapacityasymptoticallyflat_2016}. Such estimate reads  
\begin{equation}
\label{eq:bray-miao}
 \ncapa_p(\partial \Omega) \leq \left( \frac{\abs{\partial \Omega}}{4\pi}\right)^{\mathrlap{\frac{3-p}{2}}}\kern.2cm\leftidx{_2}{F}{_1}\left(\frac{1}{2}, \frac{3-p}{p-1}, \frac{2}{p-1}; 1- \frac{1}{16 \pi} \int_{\partial \Omega} \H^2 \dif \sigma\right)^{\mathrlap{-(p-1)}}\kern.7cm,
 \end{equation}
where $\kern.2cm\leftidx{_2}{F}{_1}$ denotes the hypergeometric function, see \cite{weisstein_hypergeometricfunction_}. The only needed propertis of such implicit function $\kern.2cm\leftidx{_2}{F}{_1}$ will be recalled below.  Exploiting \cref{eq:identity-masses}, we conclude in \cite[Theorem 1.3]{benatti_nonlinearisocapacitaryconceptsmass_2023} that in $\CS^1_\tau$-asymptotically flat $3$-manifolds with nonnegative scalar curvature and closed minimal boundary there holds
\begin{equation}
\label{eq:identity-pmasses}
    \ma^{\tp}_{\iso} = \ma_{\iso} = \ma_{\ADM}
\end{equation}
for any $1 < p \leq 2$. This was already known for $p = 2$ only,  under the stronger assumption of harmonic flatness at infinity \cite[Corollary 8]{jauregui_admmasscapacityvolumedeficit_2020}, meaning that $g = u^4 \delta$ outside a suitable compact set, where $u$ is harmonic with respect to the Euclidean Laplacian. 

\smallskip

Here, we observe that the same assumption of connectedness of isoperimetric sets considered in \cref{prop:misomadm} allows to get the identity $\ma^{\tp}_{\iso} = \ma_{\iso}$, for $1 < p \leq 2$, in $\CS^0$-asymptotically flat manifolds with nonnegative scalar curvature and minimal boundary. This is particularly interesting due to the fact that, in this asymptotic regime, the notion of $\ma_{\ADM}$ is not a priori available, so one cannot go through computations involving its expression. 

\begin{proposition}
Let $(M,g)$ be a complete $\CS^0$-asymptotically flat Riemannian $3$-manifold with nonnegative scalar curvature and (possibly empty) closed, minimal boundary. Assume that there exists $V_0 > 0$ such that, for any $V \geq V_0$, there exists an isoperimetric set $E_V$ of volume $V$ with connected boundary.  Then,
\begin{equation}
\label{eq:equivalencepmasses}
    \ma_{\iso} = \ma^{\tp}_{\iso}
\end{equation}
for every $ 1<p\leq 2 $.
\end{proposition}

\begin{proof} 
The inequality $\ma_{\iso} \geq \ma^{\tp}_{\iso}$ has already been pointed out to hold more in general, so we focus on the reverse one $\ma_{\iso} \leq \ma^{\tp}_{\iso}$.

\smallskip

For every $V>0$ large enough, denote $E_V$ the isoperimetric set of volume $V$. Assume that $\ma^{\tp}_{\iso} <+ \infty$ and that that $\ma_{\iso}>0$, otherwise \cite[Theorem 1.3]{benatti_isoperimetricriemannianpenroseinequality_2022} implies that $(M,g)$ is isometric to $\R^3$ and the conclusion trivially holds. Along a sequence $(V_k)_{k \in \N}$ realising the superior limit in \cref{eq:Cholimsupestimate}, there holds $\ma_{H}(\partial E_{V_k})\geq 0$. Evolve $E_{V_k}$ using the weak IMCF  and denote $E_{V_k}^t$ its sublevels. We briefly point out that, since the isoperimetric sets are not known to be homologous to the boundary, the evolving hypersurfaces may in principle touch the boundary $\partial M$. If this happens, one should consider the weak IMCF \emph{with jumps}, that is, its modification described in \cite[Section 6]{huisken_inversemeancurvatureflow_2001}.  By \cref{eq:asycomp} and the Geroch monotonicity formula, \cref{thm:geroch}, we have
\begin{equation}\label{eq:hawking_estimate}
    \ma_H(\partial E_{V_k}) \leq \lim_{t \to +\infty} \ma_H(\partial E_{V_k}^t)\leq \ma_{\iso}.
\end{equation}
Observe now that $\ma_H(\partial E_{V_k}^t)\geq 0$ and, for large $t$, the set $E_{V_k}^t$ has a connected boundary which is homologous to $\partial M$. Applying \cref{eq:bray-miao}, we have
    \begin{align}
        \ncapa_p(\partial E^t_{V_k}) &\leq \left( \frac{\abs{\partial E^t_{V_k}}}{4\pi}\right)^{\mathrlap{\frac{3-p}{2}}}\kern.2cm\leftidx{_2}{F}{_1}\left(\frac{1}{2}, \frac{3-p}{p-1}, \frac{2}{p-1}; 1- \frac{1}{16 \pi} \int_{\partial E^t_{V_k}} \H^2 \dif \sigma\right)^{\mathrlap{-(p-1)}}\kern.7cm\leq \left( \frac{\abs{\partial E^t_{V_k}}}{4\pi}\right)^{\mathrlap{\frac{3-p}{2}}},
    \end{align}
since $\leftidx{_2}{F}{_1}\left(\frac{1}{2}, \frac{3-p}{p-1}, \frac{2}{p-1}; t\right) \geq 1$ if $0\leq t \leq 1$.
In particular, this implies
\begin{equation}\label{eq:bound}
    \liminf_{t \to +\infty} \ma_{\iso}(E_{V_k}^t)\leq \liminf_{t \to +\infty}p \frac{1}{2p \pi \ncapa_p(\partial E_{V_k}^t)^{\frac{2}{3-p}}}\left( \abs{E_{V_k}^t} -\frac{4 \pi}{3} \ncapa_p (\partial E_{V_k}^t)^\frac{3}{3-p}\right) \leq p\ma^{\tp}_{\iso}.
\end{equation}
Combining \cref{eq:hawking_estimate,eq:bound}, we get that $\ma_H(\partial E_{V_k})\leq \ma_{\iso}^{\tp}<+\infty$.

Suppose now that, up to a subsequence,  $I'(V_k)^2 I(V_k) \leq 16 \pi-\varepsilon$ for some positive $\varepsilon>0$ and every $k\in \N$. Then we would have
\begin{equation}
     \ma_H(\partial E_{V_k}) \geq \frac{\varepsilon}{32 \pi }\sqrt{I(V_k)},
\end{equation}
that diverges, since the isoperimetric constant is positive, contradicting $\ma^{\tp}_{\iso}<+\infty$. Then,
\begin{equation}
   \ma_{\iso}\leq \limsup_{k\to +\infty} \frac{32 \pi\, \ma_{H}(\partial E_{V_k})}{4 \sqrt{\pi}I'(V_k)\sqrt{I(V_k)} + I'(V_k)^2 I(V_k)} =\limsup_{k\to +\infty} \ma_{H}(\partial E_{V_k})\leq \ma_{\iso}.
\end{equation}
For every $j\geq0$ there exists $E_{V_{k_j}}$ such that
\begin{equation}
    \ma_{\iso}-\frac{1}{j} \leq \lim_{t \to +\infty} \ma_{H}(\partial E_{V_{k_j}}^t) \leq \liminf_{t \to +\infty} \frac{2}{\abs{E^t_{V_k}}}\left( \abs{E^t_{V_k}} -\frac{\abs{\partial E^t_{V_k}}^{\frac{3}{2}}}{6 \sqrt{\pi}}\right) \leq \ma_{\iso}.
\end{equation}
Therefore, there exists an exhaustion made of subsets $\Omega_j= E_{V_{k_j}}^{t_j}$ of $M$ that are bounded  with connected boundaries homologous to $\partial M$, and such that
\begin{align}
    \ma_{\iso}- \frac{1}{j} \leq \ma_{H}(\partial \Omega_j) \leq \ma_{\iso} &&\ma_{\iso}- \frac{1}{j} \leq \ma_{\iso}(\Omega_j) \leq \ma_{\iso}.
\end{align}
In particular, we have
\begin{equation}
    1- \frac{1}{16 \pi}\int_{\partial \Omega_j} \H^2 \dif \sigma = \frac{4\ma_{\iso}\sqrt{\pi}}{\sqrt{ \abs{ \partial \Omega_j}}} (1+o(1)).
\end{equation}
Appealing again to \cref{eq:bray-miao} and using Taylor's expansion of $\leftidx{_2}{F}{_1}$ around $0$, we get
\begin{align}
\label{eq:bray-miao2}
\ncapa_p(\partial\Omega_j) &\leq \left( \frac{ \abs{ \partial \Omega_j}}{4\pi}\right)^{\frac{3-p}{2}}\leftidx{_2}{F}{_1}\left(\frac{1}{2}, \frac{3-p}{p-1}, \frac{2}{p-1}; \frac{4\ma_{\iso}\sqrt{\pi}}{\sqrt{ \abs{ \partial \Omega_j}}} (1+o(1))\right)^{-(p-1)} \\
&\leq \left( \frac{ \abs{ \partial \Omega_j}}{4\pi}\right)^{\frac{3-p}{2}}\left(1+ \frac{(3-p)\sqrt{\pi}}{\sqrt{ \abs{ \partial \Omega_j}}} \ma_{\iso} (1+o(1)) \right)^{-(p-1)} \\
& =\left( \frac{ \abs{ \partial \Omega_j}}{4\pi}\right)^{\frac{3-p}{2}}\left(1- \frac{(p-1)(3-p)\sqrt{\pi}}{\sqrt{ \abs{ \partial \Omega_j}}} \ma_{\iso} (1+o(1)) \right).
\end{align}
We point out now that the $p$-isocapacitary mass can indeed be computed through the equivalent formulation \cite[Proposition 5.2]{benatti_nonlinearisocapacitaryconceptsmass_2023}
\begin{equation}
\ma^{\tp}_{\iso}= \sup_{(\Omega_j)_{j \in \N}} \limsup_{j \to +\infty} \frac{2\ncapa_p(\partial \Omega)^{\frac{p-2}{3-p}}}{p(3-p)}  \left( \left(\frac{3\abs{ \Omega_j}}{4\pi}\right)^{\frac{3-p}{3}} - \ncapa_p(\partial \Omega_j)\right).
\end{equation}
Thus, \cref{eq:bray-miao2} implies
\begin{align}
    \ma^{\tp}_{\iso}& \geq \limsup_{j \to +\infty} \ma^{\tp}_{\iso}(\Omega_j) \geq \limsup_{j \to +\infty}\frac{2\ncapa_p(\partial \Omega)^{\frac{p-2}{3-p}}}{p(3-p)}  \left( \left(\frac{3\abs{ \Omega_j}}{4\pi}\right)^{\frac{3-p}{3}} - \ncapa_p(\partial \Omega_j)\right)\\
    &\geq \frac{p-1}{p} \ma_{\iso}+\limsup_{j \to +\infty}\frac{2\abs{\partial \Omega_j}^{\frac{p-2}{2}}}{p(3-p)(4\pi)^{\frac{p-2}{2}}}  \left( \left(\frac{3\abs{ \Omega_j}}{4\pi}\right)^{\frac{3-p}{3}} - \left( \frac{ \abs{ \partial \Omega_j}}{4\pi}\right)^{\frac{3-p}{2}}\right)\\
    & = \frac{p-1}{p} \ma_{\iso}+\limsup_{j\to +\infty}\frac{1}{p} \ma_{\iso}(\Omega_j) = \ma_{\iso},
\end{align}
completing the proof.
\end{proof}
For the same reasons pointed out after the proof of \cref{eq:misomadm}, the above argument provides an actual proof of \cref{eq:equivalencepmasses} for $\CS^{2}_{\tau}$-asymptotically flat manifolds with nonnegative scalar curvature and minimal boundary, when $\tau > 1/2$.
\section{Questions and open problems}
\label{sec:questions}
In this last section, we collect some natural problems and questions connected with the previous topics.
\begin{enumerate}
\item \emph{Connectedness of isoperimetric sets.} As showed above, knowing that the isoperimetric sets of big volume have connected boundaries allows us to set the equivalence among isocapacitary, isoperimetric, and $\ADM$ masses. It would be thus desirable to know if such property is true at least for $\CS^{1}_\tau$-asymptotically flat $3$-manifolds with nonnegative scalar curvature, with $\tau > 1/2$. It is likely that, suitably reworking the computations of Nerz \cite{nerz_foliationsstablespheresconstant_2015} without taking into account the behaviour of second derivatives of the metric, this can be accomplished. Some of the insights contained in \cite{benatti_isoperimetricriemannianpenroseinequality_2022} about the asymptotic behaviour of the $2$-Hawking mass and its relation with the Hawking mass, in this optimal asymptotic regime, could play a role. 
    \item \emph{Higher dimensional analysis}. All the results presented here are proved through computations that are very peculiar to dimension $3$. They are all substantially based on the application of the Gauss-Bonnet Theorem \cref{eq:gauss-bonnet} in the monotonicity calculation performed in \cref{subsec:hawking}. On the other hand, the fundamental Positive Mass Theorem has been proved through Schoen-Yau's \cite{schoen_proofpositivemassconjecture_1979}  contradiction argument up to dimension $7$ (see \cite{schoen_variationaltheorytotalscalar_1989}), and consequently also Bray's approach to the Penrose inequality \cite{bray_riemannianpenroseinequalitydimensions_2009}. It would be very interesting to understand whether, possibly with related arguments, the existence \cref{thm:isoperimetric-existence} and the various results on the isoperimetric/isocapacitary masses hold in higher dimensions.
\item \emph{A $\CS^0$-notion of $\ADM$ mass.} A weakened notion of $\ADM$ mass, resulting well posed, in particular, on $3$-manifolds that are $\CS^{0}_\tau$-asymptotically flat with $\tau > 1/2$, and that act as initial datum for a Ricci flow of metrics with nonnegative scalar curvature, has been recently devised by Burkhardt-Guim \cite{burkhardt-guim_admmassmetricsdistortion_2022}. It would be interesting to check if such quantity still coincides with the isoperimetric mass. The existence of isoperimetric sets in this class of potentially nonsmooth metrics would be of interest too. 
\item \emph{A conjecture of Huisken.} Strongly related to the previous point, we mention a famous and formidable conjecture by Huisken (see e.g. \cite[p. 2221-2223]{cederbaum_mathematicalaspectsgeneralrelativity_2021}) on an isoperimetric Positive Mass Theorem on $\CS^0$ manifolds of dimension $3$ admitting some suitable notion of nonnegative scalar curvature. We believe that the asymptotic comparison between the Hawking mass and the isoperimetric mass devised in \cite{benatti_isoperimetricriemannianpenroseinequality_2022} and discussed here may serve as a useful tool, as it strongly weakens the regularity requirements at least on the asymptotic decay of the metric. Manifolds of nonnegative scalar curvature in the  Ricci-flow-related sense considered by Burkhardt-Guim \cite{burkhardt-guim_admmassmetricsdistortion_2022}, might be a good family of metrics to test Huisken's conjecture on. 
 \end{enumerate}
\begingroup
\setlength{\emergencystretch}{1em}
\printbibliography
\endgroup

\end{document}